\documentclass[11pt,letterpaper]{amsart}

\usepackage{color}

\usepackage{graphicx, epstopdf}
\usepackage{cases}
\usepackage{bbm}
\usepackage{amssymb}
\usepackage{txfonts}
\usepackage{amscd}
\usepackage{amsfonts,latexsym,amsmath,amsxtra,mathdots,amssymb,latexsym,mathtools}
\usepackage{mathabx}
\usepackage[all,cmtip]{xy}

\usepackage{colordvi}
\usepackage{multicol}
\usepackage{hyperref}
\usepackage{tikz-cd}
\usepackage{float}
\usepackage{setspace, floatflt}
\usepackage{tensor}
\usepackage[normalem]{ulem}

\allowdisplaybreaks

 \newcommand{\BC}{{\mathbb {C}}} 
 \newcommand{\BG}{{\mathbb {G}}} 
  
\newcommand{\BM}{{\mathbb {M}}} \newcommand{\BN}{{\mathbb {N}}} 
 \newcommand{\BR}{{\mathbb {R}}} 
 \newcommand{\BV}{{\mathbb {V}}} 
\newcommand{\BY}{{\mathbb {Y}}} \newcommand{\BZ}{{\mathbb {Z}}}

 \newcommand{\CO}{{\mathcal {O}}}

\newcommand{\GL}{{\mathrm {GL}}}

\newcommand{\Hom}{\mathrm{Hom}}
  
\newcommand{\Ad}{\mathrm{Ad}}

\newcommand{\ra}{\rightarrow}

\def\fra{\mathfrak{a}}

\def\-{^{-1}}

\renewcommand{\Re}{{\mathrm{Re}\,}}

\makeatletter
\g@addto@macro\normalsize{\setlength\abovedisplayskip{3pt}}
\makeatother

\makeatletter
\g@addto@macro\normalsize{\setlength\belowdisplayskip{3pt}}
\makeatother

\newcommand{\delete}[1]{}

\theoremstyle{plain}

\newtheorem{thm}{Theorem}[section] \newtheorem{cor}[thm]{Corollary}
\newtheorem{lem}[thm]{Lemma}  \newtheorem{prop}[thm]{Proposition}
\newtheorem{prob}{Problem}[section]
 \newtheorem{defn}[thm]{Definition}

\newtheorem {rem}[thm]{Remark}

\newtheorem {as}[thm]{Assumption}

\numberwithin{equation}{section}

\newcommand{\rat}{\textup{Rat}}

\begin{document}

	\title{On local intertwining periods}

	\author{Nadir Matringe, Omer Offen and Chang Yang}
	\address{Université Paris Cité, IMJ-PRG, 8 Pl. Aurélie Nemours, 75013 Paris, France}
	\email{nadir.matringe@imj-prg.fr}
	\address{Department of Mathematics, Brandeis University, 415 South Street, Waltham, MA
		02453, USA}
	\email{offen@brandeis.edu}
	\address{Key Laboratory of High Performance Computing and Stochastic Information Processing (HPCSIP), Hunan Normal University, School of Mathematics and Statistics, Changsha, 410081, China}
	\email{cyang@hunnu.edu.cn}

	
\maketitle

\begin{abstract}
	We prove the absolute convergence, functional equations and meromorphic continuation of local intertwining periods on parabolically induced representations of finite length for certain symmetric spaces over local fields of characteristic zero, including Galois pairs as well as pairs of Prasad and Takloo-Bighash type. Furthermore, for a general symmetric space we prove a sufficient condition for distinction of an induced representation in terms of distinction of its inducing data. Both results generalize previous results of the first two named authors. In particular, for both we remove a boundedness assumption on the inducing data and for the second we further remove any assumption on the symmetric space. Moreover, when the inducing representation is uniformly bounded, we extend the field of cofficients from $p$-adic to any local field of characteristic zero. In fact this extension holds for all finite length representations under a natural generic irreducibility assumption for parabolic induction. In the case of $p$-adic symmetric spaces, combined with the necessary conditions for distinction that follow from the geometric lemma, this provides a necessary and sufficient condition for distinction of representations induced from cuspidal.
	
\end{abstract}

\section{Introduction}

	Let $F$ be either a $p$-adic field or $F=\mathbb{R}$ and let $G$ be the  group of $F$-points of a connected reductive group defined over $F$. When $F=\mathbb{R}$ we further assume that as a real reductive Lie group, the group $G$ is of inner type in the sense of \cite[2.2.8]{RRG-I}, i.e. that $\Ad(G)$ is a subgroup of the inner automorphisms of the complexification of the Lie algebra of $G$. Let $\theta$ be an $F$-involution on $G$ and $H=G^{\theta}$ the subgroup of points in $G$ fixed by $\theta$. Let 
  	\begin{align*}
 		X = \{ g\in G \mid \theta(g) = g^{-1}\}
    \end{align*}
	be the associated symmetric space equipped with the twisted $G$-action given by $g \cdot x = g x \theta(g)^{-1}$. In this note we study local intertwining periods on parabolically induced representations of $G$ associated to parabolic orbits in $X$.
	
	Let $P$ be a parabolic subgroup of $G$ with Levi subgroup $M$. An element $x\in X$ is called \emph{$M$-admissible} if $x\theta(M)x^{-1}=M$, that is, if the involution 
	\[
	\theta_x = \textup{Ad}(x) \circ \theta
	\] 
	on $G$ attached to $x$ stabilizes $M$. Fix an $M$-admissible $x\in X$. Let $\sigma$ be a smooth admissible representation of $M$ of finite length, $\ell \in \Hom_{M^{\theta_x}}(\sigma, \delta_{P^{\theta_x}}\delta_P^{-1/2})$ and $\chi$ an unramified character of $M$ anti-invariant under $\theta_x$ (that is, such that $\chi\circ \theta_x=\chi^{-1}$ on $M$). Here for a subgroup $Q$ of $G$, $\delta_Q$ is the modulus character of $Q$. Let $I_P^G(\sigma)$ be the representation of $G$ obtained by normalized parabolic induction from $\sigma$. By a local intertwining period associated to $x$ we mean the linear form on the representation $I_P^G(\sigma \otimes \chi)$ formally defined by the integral
	\begin{align*}
		\int_{P^{\theta_x} \backslash G^{\theta_x}} \ell (\varphi_{\chi} (g)) dg
	\end{align*}
	for all $\varphi_{\chi} \in I_P^G(\sigma \otimes \chi)$. We address the following question. 
	
\begin{prob}\label{prob intper}
Prove that the above integral converges absolutely in a certain sufficiently positive cone and admits a meromorphic continuation in $\chi$.
\end{prob}	
	
In order to formulate our first result we introduce the {\emph{modulus assumption}}. We say that an $M$-admissible $x\in X$ satisfies the modulus assumption if the character $\delta_{P^{\theta_x}}\delta_P^{-1/2}$ is trivial on $M^{\theta_x}$.

In Theorem \ref{thm::main-1} we solve Problem \ref{prob intper} under the modulus assumption when $F$ is $p$-adic, and under an extra natural generic irreducibility assumption (see Assumption \ref{as}) satisfied for example by unitary $\sigma$, when $F$ is real. 
In the context of either a Galois symmetric pair (that is, when $\theta$ is a Galois involution) or a pair of the Prasad and Takloo-Bighash type \cite{Nadir-Broussous-PTB-MultiOne}, the modulus assumption is automatically satisfied.

When $F$ is $p$-adic, Theorem \ref{thm::main-1} is a generalization of \cite[Theorem 1]{Matringe-Offen-InttPeriods-PLMS} where the first two named autors solve Problem \ref{prob intper} for Galois pairs under the assumption that $\sigma$ has bounded matrix coefficients. The proof of loc. cit. extends verbatim to the context of symmetric pairs where the modulus assumption is satisfied.
In order to remove the assumption that $\sigma$ has bounded matrix coefficients, however, we introduce in this work a different idea into the proof. 
In \cite{Matringe-Offen-InttPeriods-PLMS}, intertwining periods are studied inductively along a certain directed graph associated to parabolic orbits in $X$. This closely follows the strategy developed in \cite{J-L-R-periods-JAMS} and \cite{Lapid-Rogawski-periods-Galois} for global intertwining periods. The basis of induction there is a \emph{minimal} vertex in the graph. In this work we reverse the direction of induction and start from a \emph{maximal} vertex, following the strategy in \cite{Matringe-Gamma-IntertwiningPeriod-JFA} for some Galois pairs consisting of inner forms of $\GL_n$. This allows us to consider any finite length, smooth representation $\sigma$ of $M$.

For a character $\chi$ of $H$, a representation $\pi$ of $G$ is $(H,\chi)$-distinguished if it admits a non-zero $(H,\chi)$-equivariant linear form (that is, if $\Hom_H(\pi,\chi)\ne 0$). When $\chi$ is the trivial character we use $H$-distinction for $(H,\chi)$-distinction. 

In \cite[Proposition 3]{Matringe-Offen-InttPeriods-PLMS} the theory of intertwining periods is applied to deduce a sufficient condition for $H$-distinction of $I_P^G(\sigma)$ in terms of distinction of $\sigma$. In this work we observe that for the argument to work we do not need Problem \ref{prob intper} to be fully solved. In Theorem \ref{thm::distinction-sufficient} we prove in the context of a general symmetric pair that if $\sigma$ is a finite length representation of $M$ and $x\in G\cdot e\subseteq X$ is $M$-admissible and such that $\sigma$ is $(M^{\theta_x},\delta_{P^{\theta_x}}\delta_P^{-\frac12})$-distinguished then $I_P^G(\sigma)$ is $H$-distinguished. Under the modulus assumption we can prove absolute convergence on a cone of the intertwining period associated to a maximal vertex on the directed graph. A key observation is that without the modulus assumption the intertwining period can still be defined as a convergent (in a cone) iterated integral. Based on the structure of the aforementioned graph, this allows us to define a meromorphic family of $H$-invariant linear forms associated to any vertex on the graph. This is achieved by composing the intertwining period associated to a maximal vertex with a standard intertwining operator between induced representations. 
Along the way we use the fact that for finite length induced representations intertwining operators are generically invertible. For the convenience of the reader, and under Assunption \ref{as} when $F=\BR$, we provide the ingredients of the proof of this fact in Proposition \ref{prop::invert--intertwinner}.
	
	The theory of intertwining periods is a useful tool in the study of both local and global distinction problems (see for example \cite{FLO-Distinction+UnitaryGroup--IHES}, \cite{Matringe-Inner+GLn+Steinberg}, \cite{Matringe-Gamma-IntertwiningPeriod-JFA}, \cite{XueHang-Suzuki-EpsilonoDichotomy}). We refer to \cite[Remark 3]{Matringe-Offen-InttPeriods-PLMS} for a brief introduction. The results of this paper will be used in forthcoming works where we discuss the transfer of distinguished representations under the global Jacquet-Langlands correspondence. Another application we have in mind is a local analogue of the Jacquet-Rallis residue method (see \cite{Jacquet-Rallis-92-Symplectic}) as presented in \cite{Pollack-Wan-Zydor-Duke} which we hope to come back to later.
	
	The paper is organized as follows. In Section \ref{sec::notation} we introduce the relevant notation and recall some basic facts necessary for our set up. In Section \ref{sec::intt-periods} we recall the theory of intertwining periods associated to open parabolic orbits due to Brylinsky--Delorme, Carmona--Delorme and Blanc--Delorme, and that attached to closed parabolic orbits. In Section \ref{sec::maximal} we discuss the notions of maximal parabolic orbits, and maximal vertices in a graph associated to the involution $\theta$ on $G$ and introduced by Lapid and Rogawski in order to analyze parabolic orbits in $X$. Our main result on intertwining periods is proved in Section \ref{sec::main}, together with its applications to distinction of parabolically induced representations in the $p$-adic and the real case.
	
\textbf{Acknowledgments.} The authors thank the referee for his/her excellent work. Chang Yang was supported by the National Science Foundation of China (No. 12001191).

\section{Notation and preliminaries}\label{sec::notation}

Let $F$ be a local field of characteristic zero. Since complex groups can be viewed as real groups, in the archimedean case we always assume that $F=\mathbb{R}$. Let $\lvert \cdot \rvert$ be the normalized absolute value on $F$. Let $\BG$ be a connected reductive group defined over $F$ and $G = \BG(F)$. We recall that when $F=\mathbb{R}$, we moreover assume that $G$ is of inner type. For a locally compact group $Q$ let $\delta_Q$ be its modulus character.

\subsection{Symmetric spaces}\label{sec::symmetric space}

Let $\theta$ be an involution on $\BG$ defined over $F$ and $H = G^{\theta}$ be the subgroup of fixed points of $G$ under $\theta$. If $F$ is archimedean we fix a Cartan involution $\theta_c$ commuting with $\theta$ (\cite[Proposition 11.17]{Helminck-Wang-Involution}). To the symmetric pair $(G,\theta)$ we associate the symmetric space
\begin{align*}
	X = \{ g \in G \mid \theta(g) = g^{-1}  \}
\end{align*}
equipped with the twisted $G$-action given by $g \cdot x = gx\theta(g)^{-1}$, $g \in G$, $x \in X $. For $x \in X$ we denote by $\theta_x$ the involution on $G$ given by 
\begin{align*}
	\theta_x (g) = x \theta(g) x^{-1}.
\end{align*}
For any subgroup $Q$ of $G$ and $x \in X$, let $Q_x$ be the stabilizer of $x$ in $Q$. Note that 
\[
Q_x =Q^{\theta_x}= (Q \cap \theta_x(Q))^{\theta_x}.
\]

\subsection{Parabolic subgroups, roots and Weyl groups}

Start with the following conventions: if $\BY$ is an algebraic subvariety of an algebraic variety $\BV$, both defined over $F$, we will denote by $Y$ the group of $F$-points of $\BY$, and say that $Y$ is an algebraic subvariety of $V$. If $Y$ is actually an algebraic subgroup of $V$, we denote by $\mathfrak{y}$ the Lie algebra of $Y$, i.e. we use gothic letters for Lie algebras. Moreover $Y$ will inherit the same terminology as that of $\BY$ such as split torus if $\BY$ is a torus which is split over $F$. Finally if $i$ is an automorphism of $\BV$ defined over $F$, we denote by $Y^i$ the subset of 
$Y$ fixed by $i$, and by $Y^{-i}$ the subset of $Y$ on which $i$ acts as the inverse map. We use similar convetions for 
$\mathfrak{y}^i$ and $\mathfrak{y}^{-i}$ when $Y$ is an algebraic subgroup of $V$.

Fix a maximal split torus $T_0$ of $G$ which is $\theta$-stable (see \cite[Lemma 2.4]{Helminck-Wang-Involution}). If $F=\mathbb{R}$ we moreover impose that it is $\theta_c$-stable thanks to \cite[Proposition 11.18, (i)]{Helminck-Wang-Involution}. In such a situation, \cite[2.2.5]{RRG-I} and the fact that $T_0$ is maximally split imply that the subspace $\mathfrak{t}_0^{-\theta_c}$ is a maximally abelian subspace of $\mathfrak{g}^{-\theta_c}$, and it is clearly $\theta$-stable. 

We now fix a minimal parabolic subgroup $P_0$ of $G$ containing $T_0$. Let $M_0$ be the centralizer of $T_0$ in $G$. A parabolic subgroup $P$ of $G$ is called standard if it contains $P_0$. If $P$ is a standard parabolic subgroup of $G$, then it contains a unique Levi subgroup $M$ containing $M_0$. Such a Levi subgroup $M$ is called standard. Let $U$ be the unipotent radical of $P$, then $P = M \ltimes U$ is the standard Levi decomposition of $P$.

In what follows, unless otherwise specified, by a parabolic (resp. Levi) subgroup we mean a standard parabolic (resp. Levi) subgroup. By writing $P= M\ltimes U$ or $Q=L\ltimes V$ we always mean the standard Levi decomposition with $M,L$ being Levi subgroups and $U,V$ being unipotent radicals respectively.

Let $M$ be a Levi subgroup of $G$ and $T_M$ the maximal split torus in the center of $M$. Set $\fra_{M}^*  =  \rat(M) \otimes_{\BZ} \BR$ where $\rat(M)$ is the group of rational characters of $\BM$, and denote the dual space by $\fra_{M}$. Write $\fra_0$ and $\fra_0^*$ for $\fra_{M_0}$ and $\fra_{M_0}^*$ respectively. When $M \subset L$, there is a canonical direct sum decomposition $\fra_{M} = \fra_L \oplus \fra_M^L$. A similar decomposition holds for the dual space. 
For $\lambda\in \fra_0^*$ let $\lambda_M$ be the corresponding projection of $\lambda$ to $\fra_M^*$. 

For a real vector space $\fra$ we denote by $\fra_{\BC} = \fra \otimes_{\BR} \BC$ its complexification. 

Fix a maximal compact subgroup $K$ adapted to $P_0$ (see \cite[V5.1, Theorem]{Renard-book-repn}) in the $p$-adic case, and $K=G^{\theta_c}$ when $F=\mathbb{R}$. In both cases we thus have the Iwasawa decomposition $G=P_0K$ (see \cite[2.1.8]{RRG-I} when $F=\mathbb{R}$). Let $P = M\ltimes U$ be a parabolic subgroup of $G$. Let $H_M : M \ra \fra_{M}$ be the Harish-Chandra map given by
\begin{align*}
	e^{\langle \chi,H_M(m) \rangle}  = \lvert \chi (m) \rvert, \quad \chi \in \rat(M), \ m \in M.
\end{align*}
By abuse of notation we also denote by $H_M$ its unique extension  to $G= UMK$ that is left $U$-invariant and right $K$-invariant.

Let $P = M\ltimes  U \subset Q = L\ltimes V$ be two parabolic subgroups. Denote by $R(T_M,L)$ the set of roots of $T_M$ in $L$. We say that $\alpha > 0$ for $\alpha \in R(T_M,L)$ if $\alpha \in R(T_M,P \cap L)$ and $\alpha < 0$ otherwise. Recall that $R(T_0,L)$ forms a root system and let $\Delta_0^L$ be its basis of simple roots with respect to $P_0 \cap L$. Let $\Delta_M^L$ be the set of non-zero restrictions to $T_M$ of the elements of $\Delta_0^L$. The set $\Delta_M^L$ forms a basis of $(\fra_M^L)^*$. We sometimes also denote $\Delta_M^L$ by $\Delta_P^Q$. When $L =G$, we often omit the superscript $G$. Note that $R(T_0,G)$ lies in $\fra_0^*$. For every $\alpha \in R(T_0,G)$ we denote by $\alpha^{\vee} \in \fra_0$ the corresponding coroot. For $\alpha \in R(T_M,G)$, we choose $\alpha_0\in R(T_0,G)$ such that the projection of $\alpha_0$ to $\fra_{M}^*$ is $\alpha$. We then define $\alpha^{\vee}$ as the projection of $\alpha_0^{\vee}$ to $\fra_{M}$. This is independent of the choice of lifting $\alpha_0$.

Let $\rho_0 \in \fra_0^*$ be the half-sum of the positive roots of $T_0$ (counted with multiplicities). Let $\rho_P$ be the projection of $\rho_0$ on $\fra_{M}^*$. The modulus function $\delta_P$ on $P$ is then given by $e^{ \langle 2 \rho_P, H_M(\cdot) \rangle}$.

Let $W = N_G(T_0) / M_0$ be the Weyl group of $G$ with respect to $T_0$. For a Levi subgroup $M$ of $G$ let $W^M = N_M(T_0) / M_0$ be the Weyl group of $M$ with respect to $T_0$. For two Levi subgroups $M$ and $M'$ let ${}_{M'}W_M$ be the set of Weyl elements $w\in W$ that are of minimal length in $W^M w W^{M'}$. It is a complete set of representatives for the double cosets $W^{M'} \backslash W /W^M$. For two Levi subgroups $M \subset L$ let $W^L(M)$ be the set of elements $w \in W^L$ such that $w$ is of minimal length in $wW^M$ and $w M w^{-1}$ is a standard Levi subgroup of $L$. Set $W(M) = W^G(M)$. According to \cite[I.1.7, I.1.8]{Moeglin-Waldspurger-95-spectral}, one can decompose elements of $W(M)$ into products of elementary symmetries attached to simple roots in $\Delta_{M'}^G$ for Levi subgroups $M'$ of $G$ that are conjugate to $M$. In turn, this allows one to define a length function $\ell_M$ on $W(M)$. There is a unique element in 
$W^L(M)$ for which $\ell_M$ is maximal, and we denote it by $w_M^L$. Set $w_0^L=w_{M_0}^L$ and $w_0=w_0^G$.

\subsection{The involution $\theta'$}\label{sec::inlolution-theta'}

As $\theta(P_0)$ is not necessarily $P_0$, we carry out the modification as in \cite[Section 2.3]{Offen-ParabolicInduction-JNT}. Let $\tau \in W$ be the unique element such that $\theta(P_0) = \tau P_0 \tau^{-1}$. Fix a representative $\tilde\tau$ of $\tau$ in $N_G(T_0)$ and define the automorphism $\theta'$ of $G$ by $\theta'(g) = \tilde\tau^{-1} \theta(g) \tilde\tau$. Then $\theta'$ preserves $T_0$ and $P_0$. It therefore acts on $\fra_0^*$ and $\fra_0$ and preserves $\Delta_0$. Note that although $\theta'$ may not be an involution on $G$ it induces an involution on $\fra_0^*$. In what follows we continue to denote by $\theta$ and by $\theta'$ the corresponding involutions they induce on $\fra_0$ and its dual.

\subsection{Representations}\label{sec::reps}

By a representation of a reductive subgroup of $G$ we always mean a smooth complex valued representation. 

In the $p$-adic case, we consider the terms admissible and of finite length in the algebraic sense as in \cite{Bernstein-Zelevinsky}. Now we discuss in more details the case when $F$ is real. In this case, representations in this work are always SF, i.e. smooth Fréchet representations of moderate growth (see \cite{Bernstein-Krotz}). We say that a representation of $G$ is admissible (see \cite[Chapter 11]{Wallach-RRG-II} or \cite{Bernstein-Krotz}) if its subspace of $K$-finite vectors is a Harish-Chandra module, i.e. a finitely generated and amdissible $(\mathfrak{g}, K)$-module, or equivalently a $(\mathfrak{g}, K)$-module of finite length (see \cite[Section 3.3 and Theorem 4.2.1]{RRG-I}). We recall from \cite[Chapter 11]{Wallach-RRG-II} and \cite{Bernstein-Krotz} that the map assigning to an admissible representation $\sigma$ the Harish-Chandra module $\sigma_f$ of its dense subspace of K-finite vectors is an equivalence of categories, and that one can always realize $\sigma$ as the dense subspace of smooth vectors in a continuous Hilbert representation $\overline{\sigma}$. (Analogues of these results also exist in the $p$-adic case and would make some proofs in this paper more uniform, however, in the $p$-adic case we prefer to opt for a more algebraic approach.) A final but noteworthy observation is that 
an admissible representation of $G$ is irreducible if and only if it is simple as a $\mathcal{S}(G)$-module, where $\mathcal{S}(G)$ is the convolution algebra of rapidly decreasing smooth functions on $G$ defined in \cite[7.1.2]{RRG-I}. In particular admissible representations can be seen as finite length $\mathcal{S}(G)$-modules. So when $F=\mathbb{R}$ the terms admissible and of finite length can be interchanged for 
representations of $G$.

\subsection{Parabolic induction}\label{sec::para-induction+intertw-operator}

Let $P=M\ltimes U$ be a parabolic subgroup of $G$ and $\sigma$ a representation of $M$. We denote by $I_P^G(\sigma)$ the representation of $G$ by right translations on the space of functions $\varphi$ on $G$ with values in the space of $\sigma$ which satisfy
\begin{align*}
	\varphi (um g) = \delta_P(m)^{1/2} \sigma(m) \varphi(g), \quad u \in U, \ m \in M,\ g \in G,
\end{align*}
and are right invariant by some compact open subgroup of $G$ when $F$ is $p$-adic, and smooth on $G$ hen $F=\mathbb{R}$. We refer 
to \cite[Section 6.3]{Chen-Sun} for other realizations of induced representations when $F$ is real, that shall be used later. In the real case, if $\sigma$ is moreover admissible, another construction giving rise to the same space is the following (\cite{Vogan-Wallach} and \cite[Appendix A]{WallachRT} for the fact that the assumption on smooth vectors in \cite{Vogan-Wallach} is always satisfied for our representations). Take any Hilbert globalization $\overline{\sigma}$ of 
$\sigma$, then there is the usual notion of normalized parabolic induction for continuous representations on Hilbert spaces giving birth to $I_P^G(\overline{\sigma})$, and we have $I_P^G(\sigma)=I_P^G(\overline{\sigma})^\infty$ (the right hand side is the space of smooth vectors, in particular this construction does not depend on the choice of $\overline{\sigma}$). 

For $\lambda \in \fra_{M,\BC}^*$ and $\varphi \in I_P^G(\sigma)$ write $\varphi_{\lambda} (g) = e^{ \langle \lambda, H_M(g) \rangle} \varphi (g)$, the twist of $\varphi$ by $\lambda$. Let $I_P^G(\sigma,\lambda)$ be the representation of $G$ on the space of $I_P^G(\sigma)$ given by
\begin{align*}
	(I_P^G(g,\sigma,\lambda)\varphi)_{\lambda}(x) = \varphi_{\lambda} (xg).
\end{align*}
Let $\sigma[\lambda]$ denote the representation of $M$ on the space of $\sigma$ given by $\sigma[\lambda](m) = e^{\langle \lambda, H_M(m) \rangle} \sigma (m)$. The map $\varphi \mapsto \varphi_{\lambda}$ is an isomorphism of representations $I_P^G(\sigma,\lambda) \ra I_P^G(\sigma[\lambda])$.

Let $Q = L\ltimes V$ be a parabolic subgroup of $G$ containing $P$. Transitivity of parabolic induction is the natural isomorphism $F: I_P^G(\sigma) \ra I_Q^G(I_{P \cap L}^L (\sigma))$, $\varphi \mapsto F_{\varphi}$ of $G$-representations defined by
\begin{align*}
	F_{\varphi}(g)(l) = \delta_Q^{-1/2}(l)\varphi(lg),\ \ \ l\in L,\,g\in G.
\end{align*}
For $\lambda \in \fra^*_{L,\BC}$ we have the following commutative diagram of isomorphisms between representations
\begin{equation*}
	\begin{tikzcd}
		I_P^G(\sigma,\lambda) \arrow[r,"F"] \arrow[d]    & I_Q^G(I_{P \cap L}^L (\sigma),\lambda) \arrow[d]  \\
		I_P^G(\sigma [\lambda]) \arrow[r]  & I_Q^G\left( I_{P \cap L}^L (\sigma)[\lambda] \right).
	\end{tikzcd}
\end{equation*}
Here each vertical arrow is the twist by $\lambda$, an isomorphism of representations, and the bottom horizontal arrow is the unique isomorphism that makes the diagram commute. 
Explicitly, we have
\begin{align}\label{formula::holo-section-transitivity}
	(F_{\varphi})_{\lambda}(g)(l) =  e^{ -\langle \lambda + \rho_Q , H(l) \rangle} \varphi_{\lambda}(lg),\ \ \ l\in L,\,g\in G.
\end{align}

\subsection{Standard intertwining operators}
Let $P=M\ltimes U$ be a parabolic subgroup of $G$ and $(\sigma,V_\sigma)$ an \textit{admissible} representation of $M$.
Let $w \in W(M)$ and choose a representative $n$ of $w$ in $N_G(T_0)$. Let $M' = wMw^{-1}$ and $P' = M'\ltimes U'$ be the corresponding parabolic subgroup. 

Let $w\sigma$ be the representation of $M'$ on the space of $\sigma$ given by $w\sigma(m)=\sigma(n^{-1}mn)$, $m\in M'$. (The isomorphism class of this representation is independent of $n\in w$). For $c \in \BR$ set
\begin{align*}
	\mathcal{D}^{M,w} (c)= \{ \lambda \in \fra_{M}^*  \mid \langle \lambda, \alpha^{\vee} \rangle > c, \ \forall \alpha \in R(T_M,G), \alpha >0,\ w\alpha < 0  \}.
\end{align*}
There is a $c >0$ such that for $\lambda \in \fra_{M,\BC}^*$ with $\Re \lambda \in \mathcal{D}^{M,w} (c)$, and for any $g\in G$, the integral \begin{align}\label{formula::defn--intertwinner}
	( M(n,\sigma,\lambda)\varphi )_{w \lambda} (g) = \int_{U' \cap wUw^{-1} \backslash U' }  \varphi_{\lambda} (n^{-1} ug) du
\end{align} converges absolutely. In the $p$-adic case, where we refer to \cite[Section 2]{Shahidi-L-Function-AJM-81}, the convergence of the above integral means that for any $v^{\vee}$ in the smooth dual of $\sigma$, the following scalar valued integral
\begin{align*} 
	\int_{U' \cap wUw^{-1} \backslash U' }  v^{\vee}(\varphi_{\lambda} (n^{-1} ug)) du
\end{align*}
converges. In the real case, following \cite{Vogan-Wallach} which applies in our situation thanks to \cite[Appendix A]{WallachRT}, this means that the above integral converges absolutely in the sense of \textit{Bochner integrals}: in particular this requires taking a Hilbert globalization $(\overline{\sigma},H_{\overline{\sigma}})$ of $\sigma$, then the integral \eqref{formula::defn--intertwinner} is defined as an element of the Hilbert space $H_{\overline{\sigma}}$. In fact 
\cite[Lemma 1.3]{Vogan-Wallach} implies that $M(n,\sigma,\lambda)\varphi_\lambda$ is smooth, i.e. belongs to $I_{P'}^G(w\sigma)$. One checks that this in particular implies that the integral \eqref{formula::defn--intertwinner} is in fact a vector in $V_\sigma \subseteq H_{\overline{\sigma}}$. Hence both in the real and $p$-adic case (see \cite[Section 2]{Shahidi-L-Function-AJM-81} for details), when $\Re \lambda \in \mathcal{D}^{M,w} (c)$, this defines a standard intertwining operator 
\begin{align*}
	M(n,\sigma,\lambda) : I_P^G(\sigma,\lambda) \ra I_{P'}^G(w\sigma,w\lambda).
\end{align*}

The intertwining operator $M(n,\sigma,\lambda)$ admits meromorphic continuation to $\fra_{M,\BC}^*$, \cite[Theorem 2.2.2]{Shahidi-L-Function-AJM-81} when $F$ is $p$-adic and \cite[Theorem 1.13]{Vogan-Wallach} for $F$ real. In fact these references prove the following stronger fact: for any $\lambda_0\in \fra_{M,\BC}^*$, there is a rational map $R$ on $\fra_{M,\BC}^*$ such that $R(\lambda_0)M(n,\sigma,\lambda_0)$ makes sense and defines a nonzero intertwining operator from $I_P^G(\sigma,\lambda)$ to $I_{P'}^G(w\sigma,w\lambda)$.

We now record a useful result whose proof can be found in \cite[Lemma 7.2]{Matringe-Gamma-IntertwiningPeriod-JFA}, when $F$ is $p$-adic. 

\begin{lem}\label{lem::intertw'+linearform}
	Let $\ell$ be a linear form on the space of $\sigma$. Suppose that $\lambda \in \mathcal{D}^{M,w} (c)$ for $c$ large enough such that $M(w,\sigma,\lambda)$ is given by convergent integrals, and let $g \in G$ and $\varphi \in I_P^G(\sigma)$ be such that the integral
	\begin{align*}
		\int_{ wUw^{-1} \cap U' \backslash U'}  \lvert \ell(\varphi_{\lambda}(n^{-1}ug)) \rvert du
	\end{align*}
	converges. Then
	\begin{align*}
		\ell( (M(w,\sigma,\lambda)\varphi)_{w\lambda}(g) ) = \int_{wUw^{-1} \cap U' \backslash U'} \ell (\varphi_{\lambda}(n^{-1}ug) )du.
	\end{align*}
\end{lem}
\begin{proof}
	We provide the proof when $F=\mathbb{R}$. Let $U_n$ be an increasing sequence of relatively compact open subsets of the quotient $wUw^{-1}\cap U'\backslash U'$, such that their closures $C_n=\overline{U_n}$ exhaust the latter quotient. Since $C_n$ is compact and $V_{\sigma}$ is a Frechet space, the \textit{Gelfand-Pettis} integral 
		\begin{align*}
			\int_{C_n}^{\asterisk}\varphi_{\lambda}(n^{-1}ug)du \in V_{\sigma}
		\end{align*}
		exists by \cite[Theorem 3.27]{Rudin-Functional-Analysis}. By the very definition of Gelfand-Pettis integrals (see \cite[Page 77]{Rudin-Functional-Analysis}), we have
		\begin{align}\label{formula::Pettis-integral-defn}
			\ell (\int_{C_n}^{\asterisk} \varphi_{\lambda}(n^{-1}ug)du)  = \int_{C_n} \ell (\varphi_{\lambda}(n^{-1}ug)) du
		\end{align}
		for every continuous linear form $\ell$ on $V_{\sigma}$. On the other hand, the Bochner integral 
		\begin{align*}
				\int_{C_n}\varphi_{\lambda}(n^{-1}ug)du \in H_{\overline{\sigma}} 
		\end{align*} 
		is equal to $\int_{C_n}^{\asterisk} \varphi_{\lambda}(n^{-1}ug)du$ since the restriction to $V_{\sigma}$ map identifies the continuous dual of $H_{\overline{\sigma}}$ to a subspace of the continuous dual of $V_{\sigma}$. 
		
Moreover when $\lambda$ is positive enough as in the statement of the lemma, which we will assume from now on, the Bochner integral 
		\begin{align*}
				\int_{wUw^{-1} \cap U' \backslash U'}\varphi_{\lambda}(n^{-1}ug)du \in H_{\overline{\sigma}}
		\end{align*} is absolutely convergent and belongs to $V_{\sigma}$ as we already observed. 
		
		Now by our assumption, we have
		\begin{align*}
			\int_{C_n} \ell (\varphi_{\lambda}(n^{-1}ug)) du  \ra \int_{wUw^{-1} \cap U' \backslash U'} \ell (\varphi_{\lambda}(n^{-1}ug) )du.
		\end{align*}
		Thus, in view of \eqref{formula::Pettis-integral-defn}, it suffices for us to show that
		\begin{align*}
			\ell (\int_{C_n} \varphi_{\lambda}(n^{-1}ug)du)   \ra 	\ell(  \int_{wUw^{-1} \cap U' \backslash U'}  \varphi_{\lambda}(n^{-1}ug)du ).
		\end{align*}
		This would hold if we could show that 
		\begin{align*}
			\int_{C_n} \varphi_{\lambda}(n^{-1}ug)du  \ra 	\int_{wUw^{-1} \cap U' \backslash U'}  \varphi_{\lambda}(n^{-1}ug )du 
		\end{align*}
		in the Frechet topology of $V_{\sigma}$. By induction, it is enough to show that, for $X \in \mathfrak{m}_{\BC}$,
		\begin{align}\label{formula::convergence--FrechetTopo}
			\sigma(X) ( \int_{C_n} \varphi_{\lambda}(n^{-1}ug)du )  \ra \sigma(X) (  \int_{wUw^{-1} \cap U' \backslash U'}  \varphi_{\lambda}(n^{-1}ug )du  ).
		\end{align}
		This will follow from the absolute convergence of 
		\begin{align}\label{formula::abs-conv}
			 \int_{wUw^{-1} \cap U' \backslash U'} \sigma(X) (\varphi_{\lambda}(n^{-1}ug) )du  
		\end{align}
		and the following claim: for any measurable $S \subset wUw^{-1} \cap U' \backslash U'$, 
		\begin{align*}
			\sigma(X) (\int_S \varphi_{\lambda}(n^{-1}ug) du )  = \int_S \sigma(X) ( \varphi_{\lambda}(n^{-1}ug) ) du,
		\end{align*}
		where the integrals involved above are all Bochner integrals. Taking the aboslute convergence of 
		\eqref{formula::abs-conv} for granted, the computations below make sense and prove the claim:
		
		\begin{align*}
			\sigma(X) (\int_S \varphi_{\lambda}(n^{-1}ug) du ) &= \lim_{t \ra 0} \frac{1}{t} ( \sigma(e^{tX}) \int_S \varphi_{\lambda}(n^{-1}ug) du - \int_S \varphi_{\lambda}(n^{-1}ug) du )  \\
													         &= \lim_{t \ra 0} \frac{1}{t} \int_S ( \sigma(e^{tX}) \varphi_{\lambda}(n^{-1}ug) - \varphi_{\lambda}(n^{-1}ug) ) du \\
													         &= \lim_{t \ra 0} \frac{1}{t} \int_S \int_0^t \sigma(e^{sX}) \sigma(X) (\varphi_{\lambda}(n^{-1}ug)) ds du \\
													         &= \lim_{t \ra 0} \frac{1}{t}  \int_0^t \int_S \sigma(e^{sX})  \sigma(X) (\varphi_{\lambda}(n^{-1}ug)) du ds \\
													         &= \int_S \sigma(X) ( \varphi_{\lambda}(n^{-1}ug) ) du.
		\end{align*}
		Note that the second identity follows from the continuity of the action of $M$ on $H_{\overline{\sigma}}$. It remains to prove the absolute convergence of \eqref{formula::abs-conv} in order to complete the proof. 
		
		For the absolute convergence of \eqref{formula::abs-conv}, we follow the arguments in \cite[Lemma 10.1.2]{Wallach-RRG-II} where there is no action of $\sigma(X)$. Here we only indicate the necessary modifications. We can always assume $g = e$ and write
		\begin{align*}
			\varphi_{\lambda} (n^{-1} u ) = a (n^{-1} u)^{\lambda + \rho_P} \sigma (m (n^{-1} u )) \varphi ( k (n^{-1}u)),
		\end{align*}
		where $g = a(g)m(g)u(g)k(g)$ is the Iwasawa decomposition of $g$ with respect to $P$. Recall that $\sigma$ is assumed to be of moderate growth. This implies that
		\begin{align*}
			\lVert \sigma(X) (\sigma (m (n^{-1} u )) (\varphi_{\lambda}  ( k (n^{-1}u)) )) \rVert \leqslant \lVert m(n^{-1} u)\rVert^N q( \varphi_{\lambda} (k(n^{-1}u))),
		\end{align*}
		where the norm in the left-hand side is the norm on the Hilbert space and the norm in the right-hand side is the norm on the group (see \cite[Section 2.3.1]{Bernstein-Krotz}). Here $q$ is a semi-norm on $V_{\sigma}$. By definition, $\varphi_{\lambda} $ is a smooth, in particular, continuous function from $G$ to $V_{\sigma}$. As $q$ is continuous on $V_{\sigma}$, so $q (\varphi_{\lambda}(k (n^{-1}u)))$ is bounded. The rest then follows from the arguments in \cite[Lemma 10.1.2]{Wallach-RRG-II}.
		
\end{proof}

The generic invertibility of intertwining operators is crucial for the proofs of our main results. It relies on generic irreducibility of parabolically induced representations, and at the moment, we could only find a reference for this fact when $F$ is $p$-adic. Hence in the real case, we make the following assumption which should be true thanks to ongoing work of David Renard. The assumption in question is moreover satisfied when $\sigma$ is unitary and more generally uniformly bounded (see \cite[Section 10.5.1]{Wallach-RRG-II}) thanks to \cite[10.5.3]{Wallach-RRG-II}, and we only need it for $P$ maximal.

\begin{as}\label{as}(Valid for uniformly bounded $\sigma$.)
Suppose $F=\BR$, and suppose that $\sigma$ is an irreducible representation of $M$. There exists a family $(P_n)_{n\in \BN}$ of polynomial functions of the variable $\lambda \in \fra_{M,\BC}^*$ such that if $P_n(\lambda)\neq 0$ for all $n$, then 
$I_P^G(\sigma,\lambda)$ is irreducible. 
\end{as}

We recall the ingredients of the following well-known fact. 

\begin{prop}\label{prop::invert--intertwinner}
	Let $\sigma$ be a representation of $M$ of finite length. If $F$ is $p$-adic, there is $c_0 > 0$ such that, for $c > c_0$ and $\lambda \in \fra_{M,\BC}^*$ with $\Re \lambda \in \mathcal{D}^{M,w}(c)$, the intertwining operator $M(n,\sigma,\lambda)$ is an isomorphism. If $F=\BR$, under Assumption \ref{as}, the intertwining operator $M(n,\sigma,\lambda)$ is an isomorphism for $\lambda$ outside a countable union of affine hyperplanes of $ \fra_{M,\BC}^*$. 
\end{prop}
\begin{proof}
	We start with the $p$-adic case, and then explain the modifications to be made when $F=\mathbb{R}$. We make a series of reductions. Note that if $0 \ra \sigma_1 \ra \sigma \ra \sigma_2 \ra 0$ is an exact sequence of representations of $M$, we have the commutative diagram
	\begin{equation*}
		\begin{tikzcd}
			0 \arrow[r]  & I_P^G(\sigma_1,\lambda) \arrow[r] \arrow[d] & I_P^G(\sigma,\lambda) \arrow[r] \arrow[d] & I_P^G(\sigma_2,\lambda) \arrow[d] \arrow[r] & 0  \\
		    0 \arrow[r]  & I_{P'}^G(w\sigma_1,w\lambda) \arrow[r]  & I_{P'}^G(w\sigma,w\lambda) \arrow[r]  & I_{P'}^G(w\sigma_2,w\lambda) \arrow[r] & 0
		\end{tikzcd}
	\end{equation*}
where horizontal arrows are induced by the fact that parabolic induction is a functor and vertical arrows are given by intertwining operators. Since parabolic induction is exact, the rows are exact sequences. Consequently, by the five lemma and the fact that $\sigma$ is of finite length, the proposition reduces to the case where  $\sigma$ is irreducible. 

Assume that $\sigma$ is irreducible. By decomposing $w$ into elementary symmetries, we may and do assume that $w$ is an elementary symmetry (see \cite[Theorem 2.1.1]{Shahidi-L-Function-AJM-81} or \cite[Lemma I.1.8, Proposition II.1.6]{Moeglin-Waldspurger-95-spectral}). By transitivity of induction and the exactness of the functor of parabolic induction, we may and do further assume that $M$ is a maximal Levi subgroup. 

By a result of Waldspurger \cite[Theorem 3.2]{Sauvageot-density-principle}, there exists $\lambda\in \fra_{M,\BC}^*$ such that $I_P^G(\sigma,\lambda)$ is irreducible (in fact, the result asserts that the set of irreducibility points is open and dense). Let $X(M)$ be the group of unramified character of $M$ and let $\kappa: \fra_{M,\BC}^*\rightarrow X(M)$ be the surjective map given by $\lambda\mapsto e^{\langle \lambda, H_M(\cdot) \rangle}$. It now follows from \cite[Proposition VI.8.4]{Renard-book-repn}, that there is a non-empty Zariski open subset $O$ of $X(M)$ such that $I_P^G(\sigma,\lambda)$ is irreducible for all $\lambda$ with $\kappa(\lambda) \in O$. Clearly, irreducibility of $I_P^G(\sigma,\lambda)$ depends only on the projection $\lambda^G$ of $\lambda$ to the one dimensional space $(\fra_{M,\BC}^G)^*$. Consequently, $I_P^G(\sigma,\lambda)$ is irreducible for all but finitely many values of $\kappa(\lambda^G)$. 
Therefore, for $c$ large enough, for all $\lambda \in \fra_{M,\BC}^*$ with $\Re \lambda \in \mathcal{D}^{M,w}(c)$ both $I_P^G(\sigma,\lambda)$ and $I_{P'}^G(w\sigma,w\lambda)$ are irreducible. As the intertwining operator is non-zero, for such $\lambda$ it is an isomorphism.  

When $F=\BR$, the proof is essentially the same. For the reduction to irreducible $\sigma$, the best is probably to 
use the fact that finite length really means finite length as a $\mathcal{S}(M)$-module. Now we decompose $M(n,\sigma,\lambda)$ with respect to a minimal decomposition of $w$ into elementary symmetries as the product of the intertwiners 
$M(n_i,\sigma_i,\lambda_i):I_{P_i}^G(\sigma_i)\rightarrow I_{P_{i+1}}^G(s_i(\sigma_i))$. Here, $M_i$ is a translate of $M$ by an element of $W(M)$, $\sigma_i$ and $\lambda_i$ are translates of $\sigma$ and $\lambda$ by the same element, and $n_i$ corresponds to an elementary symmetry $s_{\alpha_i}$ in $W(M_i,M_{i+1})$. Set $L_i$ to be the standard Levi subgroup generated by 
$M_i$ and the root subgroups $U_{\pm \alpha_i}$, it contains $M_i$ and $M_{i+1}$ as maximal Levi subgroups. Applying transitivity of parabolic induction and making use of Assumption \ref{as}, 
one sees that each $M(n_i,\sigma_i,\lambda_i)$ is invertible outside a countable number of translates $v_{k,i}+\fra_{L_i,\BC}^*$ of $\fra_{L_i,\BC}^*$ inside $\fra_{M_i,\BC}^*$, and the result follows. 
\end{proof}


\section{Intertwining periods}\label{sec::intt-periods}

We recall the notion of intertwining periods in the generality of symmetric spaces in \cite[Section 3.1]{Matringe-Offen-InttPeriods-PLMS}.

\subsection{A general setup}

Let $(G,\theta)$ be a symmetric pair. For a representation $\pi$ of $G$ and a character $\chi$ of $H$, denote by $\Hom_{G^{\theta}}(\pi,\chi)$ the space of $(G^{\theta},\chi)$-equivariant linear forms on the space of $\pi$. We assume that $\chi$ factors through a rational character (hence trivial on any unipotent subgroup, see \cite[Corollary 14.18]{Milne-AlgGroup}).

Consider a parabolic subgroup $P=M\ltimes U$ of $G$ with a $\theta$-stable Levi subgroup $M$. By \cite[Lemma 6.3]{Offen-ParabolicInduction-JNT} one has $P^{\theta} = M^{\theta} \ltimes U^{\theta}$. Let $\sigma$ be a representation of $M$ and $\ell \in \Hom_{M^{\theta}} (\sigma,\delta_{P^{\theta}}\delta_P^{-1/2}\chi)$. For $\varphi \in I_P^G (\sigma)$, the integral
\begin{align*}
	L(\varphi) = \int_{P^{\theta} \,\backslash G^{\theta}} \chi(g)^{-1} \ell (\varphi (g)) dg
\end{align*}
makes sense formally. If it converges it defines a linear form $L \in \Hom_{G^{\theta}} (I_P^G(\sigma),\chi)$. However, in general, the integral fails to converge. One way to circumvent this technical difficulty is to introduce a complex space of unramified twists, prove convergence in a cone and meromorphic continuation.

Since $\theta$ preserves $M$ and $T_M$, it also acts as an involution on the space $\fra_{M,\BC}^*$ and on the set of roots $R(T_M,G)$. This gives rise to the decomposition
\begin{align*}
	\fra_{M,\BC}^*  = (\fra_{M,\BC}^*)^+_{\theta}  \oplus (\fra_{M,\BC}^*)^-_{\theta}
\end{align*}
where $(\fra_{M,\BC}^*)^{\pm}_{\theta}$ is the $\pm 1$-eigenspace of $\theta$ respectively.
Note that 
\begin{align*}
	\langle \lambda, H_P(m) \rangle =0, \quad \lambda \in (\fra_{M,\BC}^*)^-_{\theta},\ m \in M^{\theta}.
\end{align*}
Hence, for $\lambda \in (\fra_{M,\BC}^*)^-_{\theta}$, the integral
\begin{align}\label{formula::defn--InttP+Twist}
	L_{\lambda} (\varphi) = \int_{P^{\theta} \backslash G^{\theta}} \chi(g)^{-1} \ell (\varphi_{\lambda} (g)) dg
\end{align}
makes sense formally and defines a linear form $L_{\lambda} \in \Hom_{G^{\theta}} (I_P^G(\sigma,\lambda),\chi)$ when it converges.

For $c > 0$, set
\begin{align}\label{formula::defn--cone}
	\mathcal{D}_{M,\theta} (c) = \{ \lambda \in (\fra_{M}^*)^-_{\theta} \mid \langle   \lambda,\alpha^{\vee} \rangle > c, \  \forall \alpha \in R(T_M,P), \theta(\alpha) <0 \}.
\end{align}
This is a non-empty open set in $(\fra_{M}^*)^-_{\theta}$. In fact, it contains the projection to $(\fra_{M,\BC}^*)^-_{\theta}$ of any sufficiently large positive multiple of an element in the positive Weyl chamber of $\fra_{M}^*$ (cf. \cite[Lemma 5.2.1, (2)]{Lapid-Rogawski-periods-Galois}). 

\subsection{Open intertwining periods}\label{sec::Blanc-Delorme}

Let $P=M\ltimes U$ be a parabolic subgroup of $G$ such that $\theta(P)$ is the non-standard parabolic subgroup of $G$ opposite to $P$. That is, $M = P \cap \theta(P)$. It follows that $M$ is $\theta$-stable and $P^{\theta} = M^{\theta}$ is reductive. Assume also that $\chi = \mathbf{1}$ is the trivial character of $H$. In this case we have the following fundamental result due to Brylinsky--Delorme  \cite{Brylinski-Delorme-H-inv-form-MeroExtension-Invention} and Carmona--Delorme \cite{Carmona-Delorme-H-inv-form-FE-JFA} when $F=\mathbb{R}$, and Blanc--Delorme \cite{Blanc-Delorme} when $F$ is $p$-adic. We mention that \cite[Théorème 3]{Carmona-Delorme-H-inv-form-FE-JFA} assumes that the representation $\sigma$ hereunder is unitary, but due to recent results one can remove this assumption in the satetment below, as we explain in its proof. 

\begin{prop}\label{prop::Blanc-Delorme}
	Under the above assumptions, let $\sigma$ be a representation of $M$ of finite length. Then\\
	\textup{(1)} \quad there exists $c > 0$ such that for all $\varphi \in I_P^G(\sigma)$, $\ell \in \Hom_{M^{\theta}}(\sigma,\mathbf{1})$ and $\lambda \in (\fra_{M,\BC}^*)^-_{\theta}$ with $\Re \lambda \in \mathcal{D}_{M,\theta}(c)$ the integral \eqref{formula::defn--InttP+Twist} is absolutely convergent; \\
	\textup{(2)} \quad the linear form $\varphi \mapsto L_{\lambda}(\varphi)$ defined by \eqref{formula::defn--InttP+Twist} admits a meromorphic continuation in $\lambda \in (\fra_{M,\BC}^*)^-_{\theta}$ that we continue to denote by $L_\lambda$;\\
	\textup{(3)} \quad suppose that $L_{\lambda}$ is regular at $\lambda$. Then the linear form $L_{\lambda}$ is non-zero if and only if $\ell$ is non-zero. Moreover, if $F$ is $p$-adic and $\ell$ is non-zero, we can choose $\varphi_0 \in I_P^G(\sigma)$ with support contained in $PH$ such that $L_{\lambda}(\varphi_0)$ converges for all $\lambda$ and is identically equal to $1$.	
\end{prop}
\begin{proof}
The unitarity assumption in \cite{Carmona-Delorme-H-inv-form-FE-JFA} is there for two reasons. First for 
the absolute convergence of the integral denoted $J(P,\delta,\nu)$ there, for $\nu$ in some cone. For this part when $\sigma$ (or $\delta$ in the notations of ibid.) has finite length, we refer to the proof of Lemma \ref{lem::estimation-abs-generalizedMC} below. Second, in order to apply the properties of intertwining operators in \cite{Vogan-Wallach}, but we already observed that these properties still hold for finite length $\sigma$ thanks to \cite{WallachRT}.
\end{proof}

It was observed in \cite[Lemma 8.4]{Matringe-Gamma-IntertwiningPeriod-JFA}) that the above convergence of Blanc and Delorme remains true when the linear form $\ell$ is replaced by a more general $M^\theta$-invariant function on the space of $\sigma$ with certain boundedness conditions. As explained in loc. cit. the argument in the proof of \cite[Theorem 2.16]{Blanc-Delorme} also proves this generalization. In order to formulate it precisely, we first introduce a notion of norm $\lVert \cdot \rVert_{G^{\theta} \backslash G}$ on the quotient $G^{\theta} \backslash G$. We fix a norm $\lVert \cdot \rVert$ on $G$ as in \cite[Section I.1]{Waldspurger-Plancherel} in the $p$-adic case and as in \cite[Section 0]{Brylinski-Delorme-H-inv-form-MeroExtension-Invention} in the real case. Set $\lVert G^{\theta} g \rVert_{G^{\theta}\backslash G}  =  \lVert \theta(g)^{-1} g \rVert$, $g\in G$. 

\begin{lem}\label{lem::abs convergence--generalization-BD}
	Let $\sigma$ be a representation of $M$ of finite length. Let $\eta:V_\sigma\rightarrow \mathbb{C}$ be a function on the space of $\sigma$, which when $F=\mathbb{R}$ we assue to be continuous, such that
	\begin{align*}
		\eta(\sigma(m)v) = \eta(v), \quad m \in M^{\theta}, \ v\in V_\sigma.
	\end{align*}
	Suppose that there exists $r > 0$ such that for all $v\in V_\sigma$ there exists $C_v > 0$ such that
	\begin{align*}
		\lvert \eta (\sigma(m)v) \rvert \leqslant C_v \lVert M^{\theta} m \rVert^r_{M^{\theta} \backslash M}, \quad m \in M.
	\end{align*}
	If $F=\mathbb{R}$ we morever assume that the map $v\mapsto C_v$ is continuous. 
	Then there exists $c > 0$ such that for all $\varphi \in I_P^G(\sigma)$ and $\lambda \in (\fra_{M,\BC}^*)^-_{\theta}$ with $\Re \lambda \in \mathcal{D}_{M,\theta}(c)$, the integral
	\begin{align*}
		\int_{M^{\theta} \backslash G^{\theta}} \eta( \varphi_{\lambda}(g)) dg
	\end{align*}
	is absolutely convergent.
\end{lem}
\begin{proof}
We focus on $F=\mathbb{R}$. First the arguments in 
\cite[Lemma 4]{Brylinski-Delorme-H-inv-form-MeroExtension-Invention} show that \cite[Proposition 2.14, including (iii)]{Blanc-Delorme} also remains valid in the archimedean setting ((iii) follows from the fact that at the end of the proof of \cite[Lemma 4]{Brylinski-Delorme-H-inv-form-MeroExtension-Invention}, all $\epsilon_i(g_n)$'s converge to 
$\epsilon_i(g)$, and that one of the $\epsilon_i(g)$'s is zero). 
Now denote by $\epsilon$ the function on $G$ associated to $\eta$ in the statement of \cite[Theorem 2.16]{Blanc-Delorme}, the function $g\mapsto \epsilon(g)(v)$ is continuous on $G$ for every $v\in V_{\sigma}$ thanks to the boundedness assumption on $\eta$. In turn this implies, thanks to the continuity hypothesis on $v\mapsto C_v$, that the integral in the lemma is absolutely convergent by \cite[Equality (2.33) in the proof of Theorem 2.16]{Blanc-Delorme}.
\end{proof}

\subsection{Closed intertwining periods}

Let $P = M\ltimes U$ be a $\theta$-stable parabolic subgroup of $G$ with $M$ also $\theta$-stable. By \cite[Lemma 3.1]{Gurevich-Offen-Integrability} (the arguments work verbatim when $F = \BR$) and \cite[Lemma 6.2.4]{Springer-LAG}, $P^{\theta}$ is a parabolic subgroup of $G^{\theta}$. It follows from \cite[Corollary 2.9]{Pollack-Wan-Zydor-Duke} that $P G^{\theta}$ is closed in $G$. The associated intertwining periods in this case are simpler to handle. 
\begin{lem}\label{lem::closed}
	Let $\sigma$ be a representation of $M$, $\chi$ a character of $G^\theta$ and $$\ell \in \Hom_{M^{\theta}} (\sigma,\delta_{P^{\theta}}\delta_P^{-1/2}\chi).$$ The integral \eqref{formula::defn--InttP+Twist} converges absolutely and defines an entire in $\lambda \in (\fra_{M,\BC}^*)^-_{\theta}$ family $L_\lambda$ of $G^\theta$-invariant linear forms. Moreover, for any $\lambda \in (\fra_{M,\BC}^*)^-_{\theta}$, $L_{\lambda}$ is non-zero if and only if $\ell$ is non-zero.
\end{lem} 
\begin{proof}
We recall this classical fact when $F$ is $p$-adic first. As $P^{\theta} \backslash G^{\theta}$ is compact by above, the integral \eqref{formula::defn--InttP+Twist} converges and is entire in $\lambda$. 

In fact, as in the proof of \cite[Lemma 4.3]{Offen-ParabolicInduction-JNT}, since $PG^\theta$ is a closed subset of $G$, the representation $\operatorname{ind}_{P^\theta}^{G^\theta}(\delta_P^{\frac12}\sigma[\lambda])$ of $G^\theta$ is a quotient of $I_P^G(\sigma,\lambda)|_{G^\theta}$. Here $ \operatorname{ind}$ is non-normalized induction. 
This defines an imbedding 
\[
\Hom_{G^\theta}(\operatorname{ind}_{P^\theta}^{G^\theta}(\delta_P^{\frac12}\sigma[\lambda]),\chi)\hookrightarrow \Hom_{G^\theta}(I_P^G(\sigma,\lambda),\chi). 
\]
Furthermore the map $\ell\mapsto L_\lambda$ realizes the composition of this imbedding with the isomorphism 
\[
\Hom_{M^\theta}(\sigma[\lambda],\delta_{P^{\theta}}\delta_P^{-1/2}\chi)\simeq \Hom_{G^\theta}(\operatorname{ind}_{P^\theta}^{G^\theta}(\delta_P^{\frac12}\sigma[\lambda]),\chi)
\]
given by Frobenious reciprocity (see \cite[Proposition 4.1]{Offen-ParabolicInduction-JNT}). The lemma follows when $F$ is $p$-adic. 
The proof when $F=\mathbb{R}$ is the same. We refer to \cite[Corollary 5.4.4]{AG08} together with \cite[Proposition 6.7]{Chen-Sun} for the surjectivity of $$I_P^G(\sigma,\lambda)|_{G^\theta}\rightarrow \operatorname{ind}_{P^\theta}^{G^\theta}(\delta_P^{\frac12}\sigma[\lambda])$$
 and \cite[Remark after Theorem 6.8]{Chen-Sun} for Frobenius reciprocity.
\end{proof}


\section{Maximal parabolic orbits in a symmetric space}\label{sec::maximal}

In this section we recall the graph of involutions associated to the symmetric pair and introduce the notion of maximal vertices relative to a parabolic subgroup, which is a key to this work. Throughout this section let $P = M\ltimes U$ be a parabolic subgroup.

\subsection{Maximal twisted involutions}

In his study of twisted involutions on Weyl groups, Springer characterized in \cite[Corollary 3.4]{Springer-Involution} (resp. \cite[Proposition 3.5]{Springer-Involution}) the notion of minimal (resp. maximal) twisted involutions. In \cite[Section 3.2]{Lapid-Rogawski-periods-Galois} Lapid and Rogawski further studied minimal twisted involutions relative to a parabolic subgroup (that is, relative to a subset of simple roots). Here we introduce the analogous notion of maximal twisted involutions relative to a parabolic subgroup. 

We recall some terminology from \cite[Section 3]{Lapid-Rogawski-periods-Galois}. Consider an involution $\vartheta$ on $\fra_0^*$ that stabilizes $\Delta_0$. It also naturally acts as an involution on the sets $R(T_0,G)$ and $R^+(T_0,G)$, the Weyl group $W$, the set of Levi subgroups and the set of parabolic subgroups of $G$. 

By a twisted involution we mean $\xi \in W$ such that $\vartheta( \xi) = \xi^{-1}$. The set of twisted involutions is denoted by $\mathfrak{I}_0(\vartheta)$. A twisted involution $\xi \in {}_MW_{\vartheta M} \cap \mathfrak{I}_0(\vartheta)$ is said to be \emph{$M$-admissible} if $\xi \vartheta(M) = M$. The set of $M$-admissible twisted involutions is denoted by $\mathfrak{I}_M(\vartheta)$. If $\xi \in \mathfrak{I}_M(\vartheta)$, then $\xi \vartheta$ acts as an involution on $\fra_M^*$ and $\fra_M$. Denote by $(\fra_M^*)^{\pm}_{\xi \vartheta}$ the $\pm 1$-eigenspace of $\xi \vartheta$ in $\fra_M^*$ repectively.

To motivate our next discussion, we recall the following fact about involutions in $W$. Let $w\in W$ be an involution. The following conditions are equivalent:\\
\textup{(1)} \quad $w$ has maximal length within its conjugacy class; \\
\textup{(2)} \quad for all $\alpha \in \Delta_0$, if $w \alpha > 0$, then $w\alpha = \alpha$; \\
\textup{(3)} \quad $w = w_0w_0^L$ for some Levi subgroup $L$ such that $w\alpha = \alpha$ for all $\alpha \in \Delta_0^L$.\\
Such involutions are called maximal. 

\begin{defn}\label{defn::max-twisted-involution}
	An $M$-admissible twisted involution $\xi \in \mathfrak{I}_M(\vartheta)$ is called maximal if there exists a Levi subgroup $L \supset M$ such that $\xi = w_{\vartheta L}^G$ and $\xi \vartheta \alpha = \alpha $ for all $\alpha \in \Delta_M^L$. In this case, $L$ is uniquely determined by $\xi$ and is denoted by $L_{\xi,\vartheta}$. We denote by $\Pi_M(\vartheta)$ the set of maximal twisted involutions in $\mathfrak{I}_M(\vartheta)$.
\end{defn}

From the definition of maximality we have the following result.

\begin{prop}\label{prop::property-maximal}
	For $\xi \in \Pi_M(\vartheta)$ and $L = L_{\xi,\vartheta}$ we have $\xi \in \mathfrak{I}_L(\vartheta)$ (in particular, $L$ is $\xi \vartheta$-stable). Moreover\\
	\textup{(1)}\quad $(\fra^*_M)^-_{\xi \vartheta}  = (\fra^*_L)^-_{\xi \vartheta}$, \\
	\textup{(2)}\quad $(\fra^*_M)^+_{\xi \vartheta}  = (\fra_M^L)^*  \oplus (\fra^*_L)^+_{\xi \vartheta}$,\\
	\textup{(3)}\quad $R(T_M,G) \cap (\fra^*_M)^+_{\xi \vartheta} = R(T_M,L)$.
\end{prop}
\begin{proof}
Since $\xi$ is $M$-admissible, $\xi\vartheta$-stabilizes $M$ and therefore also $(\fra_0^M)^*$. From the definition of maximality it follows that $\xi\vartheta$ acts on $(\fra_M^L)^*$ as the identity. Consequently, $\xi\vartheta$ stabilizes $(\fra_0^L)^*$ and therefore also $L$. Let $\alpha\in \Delta_0^L$. If $\alpha\in \Delta_0^M$ then $\xi^{-1}\alpha>0$ since $\xi\in {}_MW_{\vartheta M}$. Otherwise, $\xi^{-1}(\alpha)_{\vartheta M}=\xi^{-1}(\alpha_M)=\vartheta(\alpha_M)\in R^+(T_{\vartheta M},G)$ and therefore $\xi^{-1}\alpha>0$. Since furthermore $\xi=w_{\vartheta L}^G$ it follows that $\xi\in {}_LW_{\vartheta L}$. This shows that $\xi\in \mathfrak{I}_L(\vartheta)$. The rest of the proposition is now a simple consequence of the definitions.
\end{proof}
Let $\mathfrak{G}=\mathfrak{G}_{\vartheta}$ be the weighted directed graph defined in \cite[Section 3.3]{Lapid-Rogawski-periods-Galois} associated to the involution $\vartheta$. In what follows we freely use the notion of weight introduced by Lapid and Rogawski for vertices on the graph. In analogy with Proposition 3.3.1 of ibid. we have
\begin{prop}\label{prop::max-t-involution-characterization}
	Let $\xi \in \mathfrak{I}_M(\vartheta)$. The following conditions are equivalent:\\
	\textup{(1)}\quad $\xi \in \Pi_M(\vartheta)$;\\
	\textup{(2)}\quad for all $\alpha \in \Delta_M$, if $\xi \vartheta \alpha > 0$, then $\xi \vartheta \alpha = \alpha$. \\
	\textup{(3)}\quad the vertex $(M,\xi)$ has locally maximal weight in $\mathfrak{G}$.
\end{prop}
\begin{proof}
	Suppose that $\xi \in \Pi_M(\vartheta)$ and let $L = L_{\xi,\vartheta}$. For $\alpha \in \Delta_M$, if $\xi \vartheta \alpha >0$, then we necessarily have $\alpha \in \Delta_M^L$ because $\xi = w_{\vartheta L}^G$ and $\vartheta\alpha\in \Delta_{\vartheta M}$. Hence $\xi \vartheta \alpha = \alpha$ by definition. This shows that $(1)$ implies $(2)$.
	
	Conversely, suppose that $\xi \vartheta \alpha = \alpha$ whenever $\alpha \in \Delta_M$ and $\xi \vartheta \alpha > 0$. Let $L$ be the Levi subgroup containing $M$ such that $\Delta_M^L = \{\alpha \in \Delta_M \mid \xi \vartheta \alpha = \alpha \}$. By an argument similar to the proof of Proposition \ref{prop::property-maximal} we deduce that $\xi \in \mathfrak{I}_L(\vartheta)$. Futhermore, for all $\beta \in \Delta_{\vartheta L}$ we necessarily have $\xi \beta  < 0$ and consequently $\xi=w_{\vartheta L}^G$. Indeed, assume on the contrary that $\beta \in \Delta_{\vartheta L}$ is such that $\xi\beta>0$ and let $\alpha\in \Delta_{\vartheta M}$ be such that $\beta=\alpha_{\vartheta L}$. Then $\vartheta\alpha\in \Delta_M$ is such that $\xi\vartheta(\vartheta\alpha)=\xi\alpha>0$ and therefore, by definition, $\vartheta\alpha\in \Delta_M^L$ and therefore $\beta=0$. This is a contradiction. We conclude that $(2)$ implies $(1)$. 
	
	The equivalence of $(2)$ and $(3)$ follows immediately from \cite[Lemma 3.2.1]{Lapid-Rogawski-periods-Galois}.
\end{proof}

\subsection{Admissible orbits}

Following \cite[Section 3]{Offen-ParabolicInduction-JNT} we fiber the $P$-orbits in $X$ over certain twisted involutions.  Recall that $\tau \in W$ and an involution $\theta'$ on $\fra_0^*$ which is independent of the choice of $\tilde\tau\in \tau$ were defined in Section \ref{sec::inlolution-theta'}.  The fibration map 
\[
\iota_M: P\backslash X \rightarrow ( {}_MW_{\theta'(M)} \cap \mathfrak{I}_0(\theta'))\tau^{-1}
\]
is characterized by the identity
\begin{align*}
	P x \theta(P) = P\iota_M(P \cdot x) \theta(P).
\end{align*}
Recall that $x \in X$ or its $P$-orbit $P \cdot x$ is called $M$-admissible if $\iota_M(P \cdot x) \tau$ is $M$-admissible, that is, $M = w \theta(M) w^{-1}$ with $w = \iota_M(P \cdot x)$. By \cite[Lemma 6.1]{Offen-ParabolicInduction-JNT}, $x \in X$ is $M$-admissible if and only if $x \in UN_{G,\theta}(M) \theta(U)$, where 
\[
N_{G,\theta}(M) = \{ g \in G \mid g \theta(M) g^{-1} = M\}.
\]
\begin{defn}
	An $M$-admissible $P$-orbit $\CO$ in $X$ is called maximal if $\iota_M(\CO) \tau\in \Pi_M(\theta')$.
\end{defn}
By \cite[Lemma 3.2]{Offen-ParabolicInduction-JNT}, every $M$-admissible $P$-orbit $\CO$ has a non-empty intersection with $N_{G,\theta}(M)$ . Moreover, $\CO \cap N_{G,\theta}(M) = \CO \cap M \iota_M(\CO)$ is a single $M$-orbit in $X$.

\subsection{The graph of involutions $\mathfrak{G}$}\label{sec::graph-involution}

Set
\begin{align*}
	X[M] = \{ x \in X \mid \theta_x(M) = M \} = X \cap N_{G,\theta}(M).
\end{align*}
Note that a $P$-orbit in $X$ is $M$-admissible if and only if it contains an element of $X[M]$.

Following \cite{Offen-ParabolicInduction-JNT}, we define a directed, labeled graph $\mathfrak{G}$ as follows. The vertices of $\mathfrak{G}$ are the pairs $(M,x)$, where $M$ is a standard Levi subgroup of $G$ and $x \in X[M]$. The edges of $\mathfrak{G}$ are given by
\begin{align*}
		(M,x) \stackrel{n}{\searrow} (M_1,x_1) 
\end{align*}
if there is $\alpha \in \Delta_P$ with $- \alpha \ne \theta_x(\alpha) < 0$ such that $n \in s_{\alpha} M$ where $s_{\alpha} \in W(M)$ is the elementary symmetry associated to $\alpha$, $M_1 = nMn^{-1}$ and $x_1 = n \cdot x$. Note that $(M_1)_{x_1} = n M_x n^{-1}$.

Let $(M,x)$ be a vertex in $\mathfrak{G}$ and $w = \iota_M(P \cdot x)$. Since $\theta_x$ stabilizes $M$, it also acts on the space $\fra_M^*$ and on the set of roots $R(T_M,G)$. 
This action of $\theta_x$ coincides with the action of $w\theta$ since, in fact, $x\in Mw=w\theta(M)$.
\begin{rem}
The graph $\mathfrak{G}$ refines the Lapid-Rogawski graph $\mathfrak{G}_{\theta'}$ associated to the involution $\theta'$ on $\fra_0^*$ in \cite[Section 3.3]{Lapid-Rogawski-periods-Galois}. More precisely, in the above notation $(M,x) \stackrel{n}{\searrow} (M_1,x_1) $ is a vertex in $\mathfrak{G}$ if and only if
$(M,\iota_M(x)\tau) \stackrel{\alpha}{\rightarrow} (M_1,\iota_{M_1}(x_1)\tau) $ is a vertex in $\mathfrak{G}_{\theta'}$
\end{rem}


\subsection{Maximal vertices}

\begin{defn}
	A vertex $(M,x)$ in the graph $\mathfrak{G}$ is called maximal if $\iota_M(P \cdot x) \tau \in \Pi_M(\theta')$. For a maximal vertex $(M,x)$ let $L_{M,x}=L_{\xi,\theta'}$ with $\xi=\iota_M(P \cdot x)\tau$ be the associated Levi subgroup as in Definition \ref{defn::max-twisted-involution}. 
\end{defn}
It has been shown in the proof of Proposition \ref{prop::max-t-involution-characterization} that $L_{M,x}$ can be characterized as the unique Levi subgroup containing $M$ such that $\Delta_M^L = \{ \alpha \in \Delta_M \mid \theta_x(\alpha) = \alpha \}$.
\begin{lem}\label{lem::max vertice}
	Let $(M,x)$ be a maximal vertex in $\mathfrak{G}$. Let $Q=L\ltimes V$ be the parabolic subgroup with Levi $L = L_{M,x}$ and denote by $Q^-$ the non-standard parabolic subgroup of $G$ opposit to $Q$, so that $Q\cap Q^-=L$. One has 
	\begin{align*}
		\theta_x(L) = L ,\quad \theta_x(L \cap P) = L \cap P.
	\end{align*}
	and
	\begin{align*}
		 \theta_x(Q) = Q^-.
	\end{align*}
\end{lem}
\begin{proof}
	It follows from Proposition \ref{prop::property-maximal} that $\theta_x$ stabilizes $L$ and from the definition of maximality that it acts trivially on $\Delta_M^L$ and maps every element of $\Delta_L$ to a negative root. The lemma readily follows.
\end{proof}

\begin{prop}\label{prop::maximal vertice-path}
	Let $(M,x)$ be a vertex in the graph $\mathfrak{G}$. There exists a path
	\begin{align*}
		(M_1,x_1) \stackrel{n_1}{\searrow} (M_2,x_2) \stackrel{n_2}{\searrow} \cdots \stackrel{n_k}{\searrow} (M,x)
	\end{align*}
	in $\mathfrak{G}$ such that $(M_1,x_1)$ is a maximal vertex.
\end{prop}
\begin{proof}
	If $(M,x)$ is not maximal, then by Proposition \ref{prop::max-t-involution-characterization} we can choose $\alpha \in  \Delta_M$ such that $\theta_x(\alpha) >0$ and $\theta_x(\alpha) \ne \alpha$. Taking $n \in s_{\alpha}M$ we then have an edge
	\begin{align*}
		(s_{\alpha} M s_{\alpha}^{-1}, n \cdot x) \stackrel{n^{-1}}{\searrow} (M,x)
	\end{align*} 
	in $\mathfrak{G}$. It follows from \cite[Lemma 3.2.1]{Lapid-Rogawski-periods-Galois} that the weight of $(s_{\alpha} M s_{\alpha}^{-1}, n \cdot x)$ is bigger by two than the weight of $(M,x)$. Thanks to Proposition \ref{prop::max-t-involution-characterization} the proposition follows by induction on the weight of a vertex. 
\end{proof}

Let $(M,x)$ be a vertex in the graph $\mathfrak{G}$. We define a character $\delta_x$ on $M_x$ by
\begin{align}
	\delta_x (m) = \delta_{P_x}(m)\delta_P^{-1/2}(m).
\end{align}
Such characters arise naturally in the study of distinction problems of induced representations using Mackey theory. They are in general not trivial. We note, however, that for Galois pairs and symmetric pairs of Prasad-Takloo-Bighash type, $\delta_x \equiv \mathbf{1}$ for all $(M,x)$ (see \cite[Corollary 6.9]{Offen-ParabolicInduction-JNT} and \cite[Equation (5.3) and Remark 5.4]{Nadir-Broussous-PTB-MultiOne}). To each vertex $(M,x)$ and a representation $\sigma$ of $M$, we can formally attach the integral
\begin{align*}
	\int_{P_x \backslash G_x} \ell (\varphi_{\lambda} (g)) dg,
\end{align*}
where $\ell \in \Hom_{M_x}(\sigma,\delta_x)$ and $\lambda \in (\fra_{M,\BC}^*)^-_x := (\fra_{M,\BC}^*)^-_{\theta_x}$. In the next section we will prove that if $\delta_x$ is trivial then these integrals converge absolutely for $\lambda$ in a certain cone, satisfy certain functional equations and admit meromorphic continuation (see Theorem \ref{thm::main-1}). 
																							

\section{Main results}\label{sec::main}

\subsection{}

In the following result, under a convergence assumption, we obtain functional equations for intertwining periods that are local analogues of \cite[Proposition 10.1.1]{Lapid-Rogawski-periods-Galois} in the global Galois setting (see also \cite[Proposition 33]{J-L-R-periods-JAMS}). It relates the intertwining periods attached to two adjacent vertices in the graph (cf. \cite[Proposition 8.4]{Matringe-Gamma-IntertwiningPeriod-JFA} for Galois pairs of certain inner forms of $\GL_n$). 

\begin{prop}\label{prop::intt-periods-path}
	Let $P = M\ltimes U$ and $P_1 = M_1\ltimes U_1$ be two parabolic subgroups of $G$. Assume that $(M,x) \stackrel{n}{\searrow} (M_1,x_1)$ is an edge on the graph $\mathfrak{G}$ and $\alpha \in \Delta_M$ is such that $n \in s_{\alpha}M$. Let $\sigma$ be a representation of $M$ and $\ell \in \Hom_{M_x}(\sigma,\delta_x)$. Choose $c > 0$ such that the intertwining operator $M(n,\sigma,\lambda)$ is given by an absolutely convergent integral for every $\lambda \in \fra_M^*$ such that $\langle\lambda,\alpha^\vee\rangle>c$ (and in particular, if $\lambda\in \mathcal{D}_{M,\theta_x}(c)$). Then for $\varphi\in I_P^G(\sigma)$
	we have
	\begin{align}\label{formula::intt-period-path-equality}
		\int_{P_x \backslash G_x} \ell (\varphi_{\lambda}(g)) dg = \int_{(P_1)_{x_1} \backslash G_{x_1} } \ell ((M(n,\sigma,\lambda) I_P^G(n,\sigma,\lambda) \varphi)_{s_{\alpha}\lambda} (g)) dg
	\end{align}
	in the sense that, if the left hand side of \eqref{formula::intt-period-path-equality} is absolutely convergent, then so is the right hand side and the two integrals are equal to each other.
\end{prop}
\begin{proof}
	Denote by $Q = L\ltimes V$ the parabolic subgroup of $G$ containing $P$ such that $\Delta_P^Q = \{ \alpha\}$. Thus, $P_1$ is contained in $Q$ and $\Delta_{P_1}^Q = \{ -s_{\alpha} \alpha \}$. 
	We begin with the right hand side of \eqref{formula::intt-period-path-equality} and explain how to formally derive that it equals the left hand side. The absolute convergence required to prove the proposition follows from the assumptions and an application of the Fubini theorem at the second to last equality obtained.	
	By Lemma \ref{lem::intertw'+linearform} we have	
	\begin{equation}
		\begin{aligned}\label{formula::intt-periods-path-I}
			&\int_{(P_1)_{x_1} \backslash G_{x_1} } \ell ((M(n,\sigma,\lambda) I_P^G(n,\sigma,\lambda)\varphi)_{s_{\alpha}\lambda} (g)) dg  \\
			&=  \int_{(P_1)_{x_1} \backslash G_{x_1} } \int_{U_1 \cap L} \ell (\varphi_{\lambda}(n^{-1} ugn)) dudg. 
		\end{aligned}
	\end{equation}
	By \cite[Lemma 6.4]{Offen-ParabolicInduction-JNT}, one has $V_{x_1} = n U_x n^{-1}$ and an isomorphism $V_{x_1} \backslash (U_1)_{x_1} \cong U_1 \cap L$. For fixed $g \in G$, the function $u \mapsto \ell (\varphi_{\lambda}(n^{-1}ugn))$ on $U_1$ is left invariant by $V$. Thus, the right hand side of \eqref{formula::intt-periods-path-I} equals to
	\begin{align}\label{formula::intt-periods-path-II}
		\int_{(P_1)_{x_1} \backslash G_{x_1} }\int_{V_{x_1} \backslash (U_1)_{x_1}} \ell(\varphi_{\lambda}(n^{-1}ugn)) dudg.
	\end{align}
	By \cite[Lemma 6.3]{Offen-ParabolicInduction-JNT}, one has $(P_1)_{x_1} = (M_1)_{x_1} \ltimes (U_1)_{x_1}$ and $P_x = M_x \ltimes U_x$. Then one has $nP_x n^{-1} = (M_1)_{x_1} \ltimes V_{x_1}$. Hence \eqref{formula::intt-periods-path-II} equals to
	\begin{align*}
		&\int_{(P_1)_{x_1} \backslash G_{x_1} } \int_{ nP_x n^{-1} \backslash (P_1)_{x_1}} \delta_{(P_1)_{x_1}}^{-1}(p) \ell ( \varphi_{\lambda} (n^{-1} p g n)) dp dg  \\
		& = \int_{n P_x n^{-1} \backslash G_{x_1}} \ell (\varphi_{\lambda}(n^{-1} g n )) dg  = \int_{P_x \backslash G_x} \ell (\varphi_{\lambda} (g)) dg.
	\end{align*}
	The proposition follows.
\end{proof}

We have the following lemma as an intermediate step when considering the convergence of intertwining periods in the case of maximal vertices.
\begin{lem}\label{lem::estimation-abs-generalizedMC}
	Let $P = M\ltimes U$ be a standard parabolic subgroup of $G$. Let $(M,x)$ be a maximal vertex in the graph $\mathfrak{G}$ such that $\delta_x \equiv \mathbf{1}$. Write $L = L_{M,x}$ and let $Q = L \ltimes V$ be the parabolic subgroup with Levi subgroup $L$. Let $\sigma$ be a representation of $M$ of finite length and $\ell \in \Hom_{M_x}(\sigma,\mathbf{1})$. 
	Then there exists $ r > 0$ such that for any $f \in I_{P\cap L}^L(\sigma)$, there is a constant $C _f > 0$ such that
	\begin{align}\label{formula::majorant-inner closed}
		\int_{ (P \cap L)_x \backslash L_x} \lvert \ell (f(hl)) \rvert dh \leqslant C_f \lVert L_x l \rVert_{L_x \backslash L}^r
	\end{align}
	for all $l \in L$. 
	When $F=\mathbb{R}$, one can moreover assume that $C_f$ is such that $f\mapsto C_f$ is continuous. 
\end{lem}
\begin{proof}
We first treat the $p$-adic case, and then indicate the modifications to be done when $F=\mathbb{R}$. In order that the integral in \eqref{formula::majorant-inner closed} formally makes sense, we first show that $\delta_{P_x}\delta_P^{-1/2} = \delta_{(P \cap L)_x}\delta_{P \cap L}^{-1/2}$ on $M_x$. Note that we have $P_x = ( P \cap L)_x$. Indeed, by Lemma \ref{lem::max vertice} and in its notation we have $\theta_x(Q) = Q^-$ and therefore 
	\begin{align*}
		P_x = (P \cap \theta_x(P))_x \subseteq (P\cap Q^-)_x = (P \cap L)_x.
	\end{align*}
	We have $\delta_P(m) = \delta_{P \cap L}(m) \delta_Q(m)$ for $m \in M$. Since $\theta_x(Q) = Q^-$ we have $\delta_Q(l)=\delta_{Q^-}(\theta_x(l))$, $l\in L$. Since, furthermore, $\delta_{Q^-}=\delta_Q^{-1}$ on $L$ it follows that $\delta_Q$ is trivial on $L_x$ and in particular on $M_x$. The required identity follows. The integral converges absolutely as $P \cap L$ is $\theta_x$-stable by Lemma \ref{lem::max vertice} and the quotient is thus compact by \cite[Lemma 3.1]{Gurevich-Offen-Integrability}.
	
	The argument in \cite[Lemma 8.3]{Matringe-Gamma-IntertwiningPeriod-JFA} can be applied here with little modification to obtain the desired bound. We include it here for the sake of completeness. Let $\phi \in I_{P \cap L}^L (\mathbf{1})$ be such that $\phi$ equals to $1$ on $K \cap L$. For $h \in L_x$ and $ l \in L$, we write $hl = u_1m_1k_1$ with $u_1 \in U \cap L$, $m_1 \in M$ and $k_1 \in K \cap L$. Then 
	\[
	 \lvert \ell (f(hl)) \rvert = \phi (hl) \lvert \ell (\sigma(m_1) f(k_1)) \rvert. 
	 \]
	 By the bound of Lagier in \cite[Theorem 4, (i)]{Lagier-2008} and the fact that $f(k_1)$ takes only finitely many vectors in the space of $\sigma$ as $h$ and $l$ vary, we have
	\begin{align*}
		\lvert \ell (f(hl)) \rvert \leqslant C \phi(hl) \lVert M_x m_1 \rVert^r_{M_x \backslash M}
	\end{align*}
 	for some constants $C,r>0$ independent of $h$ and $l$. We claim that
 	\begin{align*}
 		\lVert M_x m_1 \rVert_{M_x \backslash M} \leqslant C_1 \lVert L_x l \rVert_{L_x \backslash L}
 	\end{align*}
 	for some $C_1 >0$ independent of $h$ and $l$. In fact, 
 	\begin{align*}
 		\lVert L_x l \rVert_{L_x \backslash L} = \lVert L_x hl \rVert_{L_x \backslash L} \geqslant C_2 \lVert  L_x u_1 m_1 \rVert_{L_x \backslash L}
 	\end{align*}
 	for some $C_2 > 0$ by \cite[(3.9)]{Lagier-2008}. We have
 	\begin{align*}
 		\lVert L_x u_1 m_1 \rVert_{L_x \backslash L } = \lVert \theta_x(u_1 m_1)^{-1} u_1 m_1 \rVert = \lVert \theta_x(m_1)^{-1} m_1 v \rVert
 	\end{align*}
 	for some $v \in U \cap L$ as $U \cap L$ is $\theta_x$-stable by Lemma \ref{lem::max vertice} and is normalized by $M$. Now we get
 	\begin{align*}
 		 \lVert \theta_x(m_1)^{-1} m_1 v \rVert \geqslant  C_3 \lVert \theta_x(m_1)^{-1} m_1 \rVert = C_3 \lVert M_x m_1 \rVert_{M_x \backslash M},
 	\end{align*}
	for some $C_3 >0$. Here the inequality follows from \cite[Lemme II.3.1]{Waldspurger-Plancherel}. Note that we can apply directly \cite[Theorem 4, (1)]{Lagier-2008} to the integral
 	\begin{align*}
 		\int_{ (P \cap L)_x \backslash L_x}  \phi (hl) dh
 	\end{align*}
 	to obtain a bound of a similar type. The lemma now follows combining the two bounds. 
	
	When $F=\mathbb{R}$, the proof is similar up to replacing Lagier's bound by that of \cite[Theorem 5.11]{KSS} which could be reformulated in the same form as Lagier's bound. Also, noting that the quantity $q(v)$ in this bound is such that $q$ is continuous, together with the compactness of $K\cap L$, allows to skip the finite values of $k_1$ argument above. Also one notes that the constant $C$ can be chosen as $q_K(f):=\mathrm{sup}_{k \in K} q(f(k))$, which depends continuously on $f$. Later in the argument, one directly uses the majorization of relative coefficients in \cite[Theorem 5.11]{KSS} instead of that in \cite{Lagier-2008}. 
\end{proof}

\begin{thm}\label{thm::main-1}
Let $P = M\ltimes U$ be a standard parabolic subgroup of $G$, $x \in X[M]$ such that $\delta_x \equiv \mathbf{1}$, $\sigma$ a representation of $M$ of finite length and $ \ell \in \Hom_{M_x}(\sigma,\mathbf{1})$. Assume that $\sigma$ satisfies Assumption \ref{as} when $F=\BR$. Then there exists a non-empty open cone $C$ in $(\fra^*_{M})^-_x$ such that for all $\lambda \in (\fra^*_{M,\BC})_x^-$ with $\Re \lambda\in C$ the integral
	\begin{align}
		J_P^G (\varphi;x,\ell,\sigma,\lambda) = \int_{P_x \backslash G_x} \ell (\varphi_{\lambda}(g)) dg
	\end{align}
	converges absolutely for every $\varphi \in I_P^G(\sigma)$. Furthermore, the linear form $J_P^G(x,\ell,\sigma,\lambda)$ admits a meromorphic continuation to $\lambda \in (\fra^*_{M,\BC})^-_x$ and satisfies the following functional equations. Whenever $\alpha \in \Delta_P$ and $n \in s_{\alpha} M$ are such that $(M,x) \stackrel{n}{\searrow} (M_1,x_1)$ in $\mathfrak{G}$ we have
	\begin{align*}
		J_P^G(\varphi;x,\ell,\sigma,\lambda) = J_{P_1}^G ( M(n,\sigma,\lambda)I_P^G(n,\sigma,\lambda) \varphi;x_1,\ell,s_{\alpha}\sigma, s_{\alpha}\lambda).
	\end{align*}
\end{thm}
\begin{proof}
	Assume first that $(M,x)$ is a maximal vertex. We justify the absolute convergence and meromorphic continuation in this case. Let $Q = L\ltimes V$ be the parabolic subgroup with $L=L_{M,x}$. Note that $Q_x = L_x$ is unimodular and $P_x = ( P \cap L)_x$ (see Lemma \ref{lem::max vertice}). We formally have
	\begin{align}\label{formula::main-1-1}
		\int_{P_x \backslash G_x} \ell (\varphi_{\lambda} (g)) dg = \int_{Q_x \backslash G_x} \int_{ P_x \backslash Q_x}  \ell (\varphi_{\lambda} (h g )) dh dg.
	\end{align}
	Let $\eta$ be the non-negative function on $I_{ P \cap L}^L(\sigma)$ defined by
	\begin{align*}
		\eta (f) = \int_{ (P \cap L)_x \backslash L_x} \lvert \ell (f(h)) \rvert dh, \quad f\in I_{P \cap L}^L (\sigma).
	\end{align*}
	In order to apply Lemma \ref{lem::abs convergence--generalization-BD} we need to check that $\eta$ is continuous when $F=\mathbb{R}$. Recall from \cite[Section 4]{Casselman-Globalization-HC-module-89} that the topology on $I_{ P \cap L}^L(\sigma)$ is inherited from the semi-norms $N$ defining the topology of $\sigma$ by associating to them 
the semi-norms $N_{X}(f)=\mathrm{sup}_{k\in K_L} N(R_Xf(k))$ where $K_L=K\cap L$, $X \in U(\mathfrak{l})$ and $R_X$ is the action of the universal enveloping algebra $U(\mathfrak{l})$ of $L$. Then the continuity of $\eta$ follows from the inequality 
\[|\eta(f_1)-\eta(f_2)|\leq \int_{(P\cap L)_x\backslash L_x} |\ell(f_1(h)-f_2(h))|dh,\] together with the continuity of $\ell$ on $V_\sigma$ which holds by definition of distinction in the archimedean case. 
	By Lemma \ref{lem::estimation-abs-generalizedMC}, there exists $r > 0$ such that
	\begin{align*}
		 \eta (I_{ P \cap L}^L(l,\sigma) f) \leqslant C_f \lVert L_x l \rVert^r_{L_x \backslash L},\ \ \ l\in L.
	\end{align*}
As $\eta$ is $L_x$-invariant, by Lemma \ref{lem::abs convergence--generalization-BD}, there exists $c > 0$ such that for all $F_{\varphi}  \in I_Q^G\left(I_{P \cap L}^L \sigma \right)$ and $\lambda \in \mathcal{D}_{L,x}(c)$, the integral
	\begin{align*}
		\int_{Q_x \backslash G_x} \eta((F_{\varphi})_{\lambda} (g) ) dg 
	\end{align*}
	converges. Here $F_{\varphi}$ is the image of $\varphi \in I_P^G(\sigma)$ under the isomorphism $F$ in Section \ref{sec::para-induction+intertw-operator}. Note that $\mathcal{D}_{L,x}(c) = \mathcal{D}_{M,x}(c)$ by Proposition \ref{prop::property-maximal} and that we have 
	\begin{align*}
		\int_{Q_x \backslash G_x} \eta((F_{\varphi})_{\lambda} (g) ) dg  = \int_{Q_x \backslash G_x} \int_{ (P \cap L)_x \backslash L_x}  \lvert \ell (\varphi_{\lambda} (hg)) \rvert dh dg
	\end{align*}
	by \eqref{formula::holo-section-transitivity}. Thus, the absolute convergence follows by Fubini's theorem. Define $\Lambda_{\ell} \in \Hom_{L_x}(I_{ P \cap L}^L \sigma,\mathbf{1})$ by
	\begin{align}\label{formula::defn-Gamma-ell}
		\Lambda_{\ell} (f) = \int_{ (P \cap L)_x \backslash L_x} \ell (f(h)) dh.
	\end{align}
	In view of \eqref{formula::main-1-1}, now well defined by a convergent integral, we have
	\begin{align*}
		J_P^G(\varphi;x,\ell,\sigma,\lambda) = J_Q^G(F_{\varphi};x,\Lambda_{\ell},I_{P \cap L}^L \sigma,\lambda).
	\end{align*}
	Therefore the meromorphic continuation follows from Proposition \ref{prop::Blanc-Delorme}. 
	
	For the proof of convergence in the general case as well as the functional equations, let $(M,x)$ be a vertex such that $\delta_x\equiv 1$ in $\mathfrak{G}$ and apply Proposition \ref{prop::maximal vertice-path}  to obtain a path
	\begin{align*}
		(M_1,x_1) \stackrel{n_{\alpha_1}}{\searrow} (M_2,x_2) \stackrel{n_{\alpha_2}}{\searrow} \cdots \stackrel{n_{\alpha_k}}{\searrow} (M_{k+1},x_{k+1})=(M,x)
	\end{align*}
	with $n_{\alpha_i} \in s_{\alpha_i} M_i$ for each $i=1,\dots,k$ and $(M_1,x_1)$ a maximal vertex. By \cite[Corollary 6.5]{Offen-ParabolicInduction-JNT} we also have $\delta_{x_1}\equiv 1$. Write $w_i = s_{\alpha_k}s_{\alpha_{k-1}} \cdots s_{\alpha_i}$ and $\lambda_i=w_i^{-1}(\lambda)$. Then for each edge we have an intertwining operator $M(n_i,w_i^{-1}\sigma,\lambda_i)$ which takes $I_{P_i}^G(w_i^{-1}\sigma,\lambda_i)$ into $I_{P_{i+1}}^G( w_{i+1}^{-1} \sigma, s_{\alpha_i}\lambda )$, where $P_i$ is the standard parabolic subgroup with Levi subgroup $M_i$. Suppose that $F$ is $p$-adic for a moment, by Proposition \ref{prop::invert--intertwinner} we can choose $c > 0$ large enough such that when $\lambda_i \in \mathcal{D}^{M_i,s_{\alpha_i}} (c)$ the intertwining operator $M(n_i,w_i^{-1}\sigma,\lambda_i)$ is an isomorphism for every $i$. Note that by our definition of an edge in the graph $\mathcal{D}_{M_i,x_i} (c) \subset \mathcal{D}^{M_i, s_{\alpha_i}} (c) \cap (\fra_{M_i}^*)^-_{x_i}$. By \cite[Lemma 5.2.1, (1)]{Lapid-Rogawski-periods-Galois} we have $s_{\alpha_i} \mathcal{D}_{M_i,x_i} (c) \subset \mathcal{D}_{M_{i+1},x_{i+1}}(c)$. Now we take $C = w_1 \mathcal{D}_{M_1,x_1}(c)  \subset \mathcal{D}_{M,x}(c) \subset (\fra^*_M)^-_x$. The required absolute convergence and functional equations follow directly from Proposition \ref{prop::intt-periods-path}. When $F=\BR$ more work is required as we need might need to exclude some affine hyperplanes to assume irreducibility of induced representations even for a sufficiently positive parameter. Let $L_i$ be the Levi subgroup generated by $M_i$ and the root subgroups $U_{\pm \alpha_i}$, in which $M_i$ and $M_{i+1}=s_{\alpha_i}(M_i)$ are maximal. Then $\mathfrak{a}_{L_i}^*$ is a hyperplane of $\mathfrak{a}_{M_i}^*$. From the proof of Proposition 2.3., the countable number of hyperplanes for which $M(n_i,w_i^{-1}\sigma,\lambda_i)$ could be non invertible are translates 
	of $\mathfrak{a}_{L_i}^*$, i.e. of the form $v+\mathfrak{a}_{L_i}^*$ for $v\in \mathfrak{a}_{M_i}^*$. Now 
	if $(\mathfrak{a}_{M_i}^*)_{x_i}^-$ was contained in such a hyperplane, it would be contained in 
	$\mathfrak{a}_{L_i}^*$. However, this would imply that $\mathfrak{a}_{L_i}^*$ is stable under $\theta_{x_i}-Id$, hence under 
	$\theta_{x_i}$. This contradicts the graph condition $\theta_{x_i}(\alpha_i)\neq -\alpha_i$ but $\theta_{x_i}(\alpha_i)<0$. Following the proof of absolute convergence in the $p$-adic case we now know that the intertwining periods converge at least outside a countable number of affine hyperplanes of $(\mathfrak{a}_{M,\BC}^{\ast})_x^- \cap \mathcal{D}_{M,x}(c)$ for some large enough $c$. However, if $J_P^G (\varphi;x,\ell,\sigma,\lambda_1)$ and $J_P^G (\varphi;x,\ell,\sigma,\lambda_2)$ converge absolutely, then $J_P^G (\varphi;x,\ell,\sigma,\lambda)$ converges for $\Re  \lambda$ in the segment joining $\Re \lambda_1 $ and  $\Re \lambda_2 $ by convexity of the real exponential map. This proves the absolute convergence of $J_P^G (\varphi;x,\ell,\sigma,\lambda)$ in some positive enough cone $\mathcal{D}_{M,x}(c)$.

As for the meromorphic continuation, if $(M,x)$ is a minimal vertex it follows from the arguments in \cite[Corollary 1]{Matringe-Offen-InttPeriods-PLMS}\footnote{The assumption that $(G,\theta)$ is Galois is not needed, only the assumption that $\delta_x\equiv 1$ is used.}. The arguments work equally well when $F = \BR$ once taken into account of some continuity properties established in \cite[Lemma 7, Theorem 2]{Brylinski-Delorme-H-inv-form-MeroExtension-Invention}. For a general vertex we may now use the functional equations obtained in Proposition \ref{prop::intt-periods-path} (combined with \cite[Corollary 6.5]{Offen-ParabolicInduction-JNT}) to deduce the meromorphic continuation from the minimal vertex case. 
\end{proof}

\subsection{Sufficient condition for distinction}

In this section we generalize \cite[Proposition 3]{Matringe-Offen-InttPeriods-PLMS} to the context of any symmetric pair and remove the boundedness assumption there for the inducing data. Let $P=M\ltimes U$ be a parabolic subgroup of $G$.
\begin{thm}\label{thm::distinction-sufficient}
	Let $\sigma$ be a finite length representation of $M$, which is assumed to satisfy Assumption \ref{as} when $F = \BR$. If there exists $x \in X$ such that $\theta_x(M) = M$ and $\sigma$ is $(M_x,\delta_x)$-distinguished, then $I_P^G(\sigma,\lambda)$ is $G_x$-distinguished for all $\lambda \in (\fra^*_{M,\BC})^-_x$.
\end{thm}


\begin{proof}
	 We prove by reverse induction on the weight of $(M,x)$ that for $0 \ne \ell \in \Hom_{M_x} (\sigma, \delta_x )$ there exists a non-zero meromorphic family of linear forms $j_P^G(x,\ell,\sigma,\lambda)$ such that the linear form $j_P^G(x,\ell,\sigma,\lambda)$ lies in $\Hom_{G_x} (I_P^G(\sigma,\lambda),\mathbf{1})$ if holomorphic at $\lambda \in (\fra_{M,\BC}^*)^-_x$.

	 For the base of induction assume that $(M,x)$ is a maximal vertex in $\mathfrak{G}$. Let $Q = L \ltimes V$ be the associated parabolic subgroup of $G$ with $L=L_{M,x}$. Define 
	\begin{align}\label{formula::distinction-1}
		j_P^G(\varphi;x,\ell,\sigma,\lambda) = \int_{Q_x \backslash G_x} \int_{ (P \cap L)_x \backslash L_x} \ell (\varphi_{\lambda} (hg)) dh dg,
	\end{align}
	where the double integral is understood as an iterated integral. We showed in the proof of Theorem \ref{thm::main-1} that the iterated double integral formally makes sense.  Define $\Lambda_{\ell} \in \Hom_{L_x}(I_{P \cap L}^L \sigma, \mathbf{1})$ as in \eqref{formula::defn-Gamma-ell}. Then as we have already seen, also in the proof of Theorem \ref{thm::main-1},
	\begin{align}\label{formula::distinction-2}
		j_P^G(\varphi;x,\ell,\sigma,\lambda) = \int_{Q_x \backslash G_x} \Lambda_{\ell} ((F_{\varphi})_{\lambda}(g)) dg.
	\end{align}
	Therefore by Proposition \ref{prop::Blanc-Delorme}, the integral in \eqref{formula::distinction-2} converges absolutely in a cone and admits a meromorphic continuation in $\lambda \in (\fra_{M,\BC}^*)^-_x$. Since $\ell\ne 0$, $\Lambda_{\ell}$ is non-zero by Lemma \ref{lem::closed} and therefore the meromorphic family $j_P^G(x,\ell,\sigma,\lambda)$ is non-zero by Proposition \ref{prop::Blanc-Delorme}. 	
	
	Assume now that $(M,x)$ is not maximal. By Proposition \ref{prop::max-t-involution-characterization} we choose $\alpha \in \Delta_M$ such that $\theta_x(\alpha) > 0$ and $\theta_x(\alpha) \ne \alpha$. Take $n \in s_{\alpha} M$ and let $M(n,\sigma,\lambda)$ be the intertwining operator taking $I_P^G(\sigma,\lambda)$ into $I_{P_1}^G(s_{\alpha}\sigma,s_{\alpha} \lambda)$, where $P_1$ is the parabolic subgroup with Levi subgroup $s_{\alpha} M s_{\alpha}^{-1}$. By the induction hypothesis there is a non-zero meromorphic family of $G_{n \cdot x}$-invariant linear forms $j_{P_1}^G(n\cdot x,s_{\alpha}\sigma,s_\alpha\lambda)$ on $I_{P_1}^G(s_{\alpha}\sigma, s_\alpha\lambda)$. Set $j_P^G(x,\ell,\sigma,\lambda)= j_{P_1}^G(n\cdot x,s_{\alpha}\sigma,s_{\alpha}\lambda) \circ M(n,\sigma,\lambda)$. It follows from the induction hypothesis and Proposition \ref{prop::invert--intertwinner}, and also by the proof of Theorem \ref{thm::main-1} when $F = \BR$, that this defines a non-zero meromorphic family of $G_x$-invariant linear forms as desired. This completes our induction. 

Now when $F$ is $p$-adic, by taking the leading term at $\lambda$ of the restriction of $j_P^G(x,\ell,\sigma,\lambda)$ to a complex line through $\lambda$ in general position, we then get a non-zero $G_x$-invariant linear form on $I_P^G(\sigma,\lambda)$. When $F=\mathbb{R}$ the following extra technicality occurs: one first considers the restriction of $j_P^G(x,\ell,\sigma,\lambda)$ to the space of $K$-finite vectors $I_P^G(\sigma,\lambda)_f$, which is nonzero by density of these vectors and continuity of $j_P^G(x,\ell,\sigma,\lambda)$ for regular $\lambda$. Then the leading term argument produces a non-zero $G_x$-invariant linear form on $I_P^G(\sigma,\lambda)_f$, which extends to a non-zero $G_x$-invariant continuous linear form on $I_P^G(\sigma,\lambda)_f$ by \cite[Théorème 2]{Brylinski-Delorme-H-inv-form-MeroExtension-Invention}. The theorem follows.
\end{proof}
Finally, when $F$ is $p$-adic, combined with the geometric lemma we deduce a distinction criterion for representations induced from cuspidal.
\begin{cor}
Let $\sigma$ be a finite length, cuspidal representation of $M$ and $x\in X$. The representation $I_P^G(\sigma)$ is $G_x$-distinguished if and only If there exists $y \in G\cdot x$ such that $\theta_y(M) = M$ and $\sigma$ is $(M_y,\delta_y)$-distinguished.
\end{cor}
\begin{proof}
The `if' part is an immediate consequence of Theorem \ref{thm::distinction-sufficient}. The `only if' part follows from \cite[Theorem 4.2]{Offen-ParabolicInduction-JNT} and the cuspidality of $\sigma$.
\end{proof}

\begin{rem}
When $F=\mathbb{R}$, Theorem \ref{thm::distinction-sufficient} shows that the necessary conditions for distinction of standard modules obtained in \cite[Theorem 1.3]{Kemarsky-dist} (or rather in its proof) and \cite[Theorem 1.2]{Tamori-Suzuki} are also sufficient.
\end{rem}



\end{document}